\documentclass[12pt]{amsart}
\usepackage[margin=1.2in]{geometry}
\usepackage{graphicx,latexsym}
\usepackage{amsfonts, amssymb, amsmath, amsthm}

\newtheorem{theorem}{Theorem}[section]

\newtheorem{lemma}[theorem]{Lemma}
\newtheorem{corollary}[theorem]{Corollary}

\theoremstyle{definition}
\newtheorem{definition}[theorem]{Definition}

\theoremstyle{remark}
\newtheorem{remark}[theorem]{Remark}
\newtheorem{example}[theorem] {Example}

\begin{document}

\title[Boundary values of holomorphic functions and heat kernel method]{Boundary values of holomorphic functions and heat kernel method in translation-invariant distribution spaces}

\author[P. Dimovski]{Pavel Dimovski}
\address{Faculty of Technology and Metallurgy, Ss. Cyril and Methodius University, Ruger Boskovic 16, 1000 Skopje, Macedonia}
\email{dimovski.pavel@gmail.com}

\author[S. Pilipovi\'{c}]{Stevan Pilipovi\'{c}}
\address{Department of Mathematics and Informatics, University of Novi Sad, Trg Dositeja Obradovi\'ca 4, 21000 Novi Sad, Serbia}
\email {stevan.pilipovic@dmi.uns.ac.rs}

\author[J. Vindas]{Jasson Vindas}
\address{Department of Mathematics, Ghent University, Krijgslaan 281 Gebouw S22, 9000 Gent, Belgium}
\email{jvindas@cage.UGent.be}

\thanks{J. Vindas gratefully acknowledges support by Ghent University, through the BOF-grant 01N010114.}


\begin{abstract} We study boundary values of holomorphic functions in translation-invariant distribution spaces of type $\mathcal{D}'_{E'_{\ast}}$. New edge of the wedge theorems are obtained. The results are then applied to represent $\mathcal{D}'_{E'_{\ast}}$ as a quotient space of holomorphic functions. We also give representations of elements of $\mathcal{D}'_{E'_{\ast}}$ via the heat kernel method. Our results cover as particular instances the cases of boundary values, analytic representations, and heat kernel representations in the context of the Schwartz spaces $\mathcal{D}'_{L^{p}}$, $\mathcal{B}'$, and their weighted versions.

\end{abstract}

\maketitle

%
%
%

\section{Introduction}

The purpose of this article is to study boundary values of holomorphic functions in tube domains and solutions to the heat equation in the upper half-space in the class of translation-invariant distribution spaces of type $\mathcal{D}'_{E'_{\ast}}$, recently introduced by the authors in \cite{dpv}. These spaces are natural generalizations of the Schwartz spaces $\mathcal D'_{L^p}$ \cite{S} and their weighted versions  \cite{O1990,O}. Our considerations provide extensions of the classical theory of boundary values and analytic representations in weighted $\mathcal D'_{L^p}$ spaces, but also lead to many new results for these particular cases. We will prove a heat kernel characterization of $\mathcal{D}'_{E'_{\ast}}$, which appears to be new even for $\mathcal D'_{L^p}$.

The study of boundary values of holomorphic functions in distribution and ultradistribution spaces has shown to be quite important for a deeper understanding of properties of generalized functions which are of great relevance to the theory of PDE \cite{hormander,sato-kawai}. There is a vast literature on the subject, we refer to the book by Carmichael and Mitrovi\'{c} \cite{C-M} and references therein for an account on results concerning boundary values in distribution spaces. 

The representation of the Schwartz spaces $\mathcal{D}'_{L^{p}}$ as boundary values of holomorphic functions has also attracted much attention. See, e.g., the classical works \cite{bengel,l-z61}. More recently \cite{f-g-g}, Fern\'{a}ndez, Galbis, and G\'{o}mez-Collado have obtained various ultradistribution analogs of such results; they also obtained the representation of $\mathcal{D}'_{L^{p}}$ for $p=1,\infty$. All these works basically deal with holomorphic functions in tube domains whose bases are the orthants of $\mathbb{R}^{n}$. In a series of papers \cite{carmichael70,carmichael83,carmichael85,carmichael87}, Carmichael systematically studied boundary values in $\mathcal{D}'_{L^{p}}$ of holomorphic functions defined in more general tubes, namely, tube domains whose bases are open convex cones. The present work makes a thorough analysis of boundary values in the space $\mathcal{D}'_{E'_{\ast}}$. Many of the results we obtain in Sections \ref{boundary values}--\ref{boundary isomorphism} are new or improve earlier results even for the special case $\mathcal{D}'_{E'_{\ast}}=\mathcal{D}'_{L^{p}}$, especially in the non-reflexive cases $p=1$ or $p=\infty$.

In his seminal work \cite{M1990} Matsuzawa introduced the so-called heat kernel method in the theory of generalized
functions. His approach consists in describing distributions and hyperfunctions in terms of solutions to the heat equation fulfilling suitable growth estimates. Several authors have investigated characterizations of many others distribution, ultradistributions, and hyperfunction spaces \cite{CK,KCK,suwa, suwa-yoshino}. Our results from Section \ref{heat} add new information to Matsuzawa's program by obtaining the description of $\mathcal{D}'_{E'_{\ast}}$ via the heat kernel method. In the case of $\mathcal{D}'_{L^{p}}$, this characterization reads as follows: $f\in \mathcal{D}'_{L^{p}}$ if and only if there is a solution $U$ to the heat equation on $\mathbb{R}^{n}\times (0,t_{0})$ such that $\sup_{t\in(0,t_0)} t^{k}\|U(\:\cdot\:,t)\|_{L^{p}}<\infty$ for some $k\geq0$ and $f=\lim_{t\to0^{+}}U(\:\cdot\:,t)$.

The paper is organized as follows. Section \ref{preli} is of preliminary character, we give there a summary of properties of translation-invariant Banach spaces of tempered distributions and their associated distribution spaces $\mathcal{D}'_{E'_{\ast}}$; we also fix the notation concerning tubes and cones. Section \ref{boundary values} is devoted to the study of boundary values of holomorphic functions in $\mathcal{D}'_{E'_{\ast}}$. Our first main result (Theorem \ref{thbvw}) characterizes those holomorphic functions in truncated wedges which have boundary values in $\mathcal{D}'_{E'_{\ast}}$. It is worth pointing out that this result improves earlier knowledge about boundary values in $\mathcal{D}'_{L^{p}}$; in fact, part of our conclusion is strong convergence in $\mathcal{D}'_{L^{p}}$, $1\leq p \leq \infty$. The strong convergence was only known for $1<p<\infty$ and for certain tubes \cite{bengel,carmichael70,carmichael83,carmichael87}. Next, we consider extensions of Carmichael's generalizations of the $H^{p}$ spaces \cite{carmichael83,carmichael85,carmichael87}. In Section \ref{analytic representations} we show that every element of $\mathcal{D}'_{E'_{\ast}}$ admits a representation as a sum of boundary values. We provide in Section \ref{edge of the wedge} new edge of the wedge theorems of Epstein,
Bogoliubov, and Martineau types. Our ideas are then applied in Section \ref{boundary isomorphism} to exhibit an isomorphism between $\mathcal{D}'_{E'_{\ast}}$ and a quotient space of holomorphic functions, this quotient space is constructed in the spirit of hyperfunction theory. The paper concludes with the heat kernel characterization of $\mathcal D'_{E'_*}$ in Section \ref{heat}.
\section{Preliminaries}
\label{preli}

In this section we fix the notation and collect some notions that will be needed in the article. In particular, we give a short overview of properties of translation-invariant Banach spaces of tempered distributions and their associated distribution spaces of type $\mathcal{D}'_{E'_{\ast}}$ introduced in \cite{dpv}.  We use the standard notation from distribution theory \cite{S,V}. The distribution $\check{g}$ denotes the reflection $\check{g}(x)=g(-x)$.
If $h\in\mathbb{R}^{n}$, then $T_{h}$ is the translation operator,  $(T_{h}g)(x)=g(x+h)$. As usual,  
a subspace $Y\subset\mathcal{D}'(\mathbb{R}^{n})$ is called translation-invariant if $T_{h}(Y)= Y$ for all $h\in\mathbb{R}^{n}$. Given a locally convex space $X$, the space of $X$-valued tempered distributions is $\mathcal{S}'(\mathbb{R}^n,X)=L_b(\mathcal{S}(\mathbb{R}^n),X)$, i.e., the space of continuous linear mappings from $\mathcal{S}(\mathbb{R}^n)$ to $X$ equipped with the strong topology. The symbol ``$\hookrightarrow$'' in the expression $Y\hookrightarrow X$ means dense and continuous linear embedding.

\subsection{Tubes and cones} We start by fixing the notation concerning tubes and cones. Let $V\subseteq\mathbb{R}^{n}$ be an open subset. The tube domain
$T^{V}\subseteq \mathbb{C}^{n}$, with base $V$, is defined as
$$
T^{V}=\mathbb{R}^{n}+iV=\left\{x+iy\in \mathbb{C}^{n}:\: x\in\mathbb{R}^{n},\: y\in V \right\}.
$$
We always write $z=x+iy\in\mathbb{C}^{n}$ (and similarly for other complex variables), where $x,y\in\mathbb{R}^{n}$. We employ the notation $d_{V}(y)=\operatorname*{dist}(y,\partial V)$ for $y\in V$. The convex hull of a set $A\subset \mathbb{R}^{n}$ is denoted by $\operatorname{ch}( A)$.

Let $C\subseteq \mathbb{R}^{n}$ be an open cone (with vertex at the origin hereafter). Note that $C$ may be $\mathbb{R}^{n}$.  If $r>0$, we write in short $C(r):=C\cap \{y\in\mathbb{R}^{n}:|y|< r\}$. We denote by $\operatorname*{pr}C$ the intersection of the cone $C$ with the unit sphere of $\mathbb{R}^{n}$. We say that the subcone $C'$ is compact in $C$ and write $C'\Subset C$ if $\overline{\operatorname*{pr}C'}\subset \operatorname*{pr}C$.
It should be noticed that $d_{C}$ is homogeneous of degree 1, namely, $d_{C}(\lambda y)=\lambda d_{C}(y)$, for every $\lambda>0$.
Recall \cite{V} that the conjugate cone of $C$ is defined as
$
C^{\ast}:= \left\{ \xi\in \mathbb{R}^{n}:\ y\cdot \xi \geq 0,\: \forall y \in C\right\}.
$
Since $C$ is open, one actually has $y\cdot \xi>0$, for all $y\in C$ and $\xi\in C^{\ast}$.  For convex cones one has
$$
d_{C}(y)=\min_{\xi\in \operatorname*{pr}C^{\ast}} y\cdot \xi ,  \  \ \ y\in C.
$$
(This equality is well-known \cite[p. 61]{V}.)  The cone $C$ is called acute if $\operatorname*{int}C^{\ast}\neq\emptyset$. Given $a\geq 0$, we denote the closed Euclidean ball (centered at the origin) of radius $a$ as $\overline{B}(a)$.

\subsection{Translation-invariant Banach spaces of tempered distributions}\label{the space E}
Following \cite{dpv}, a Banach space $E$ is called a \emph{translation-invariant Banach space of tempered distributions} if it satisfies the following three conditions:
\begin{itemize}
     \item [$(a)$] $\mathcal{S}(\mathbb{R}^n)\hookrightarrow E\hookrightarrow \mathcal{S}'(\mathbb{R}^n)$.
      \item[$(b)$] $E$ is translation-invariant.
		
        \item [$(c)$] There are $M'>0$ and $\tau\geq 0$ such that
\begin{equation}\label{eqw}
\omega(h):= ||T_{-h}||_{L(E)}\leq M' (1+|h|)^\tau, \ \ \ \mbox{for all } h\in \mathbb{R}^n.
 \end{equation}
\end{itemize}
(Note that the continuity of every $T_{h}:E\to E$ is an immediate consequence of $(a)$, $(b)$, and the closed graph theorem.)

These three axioms imply \cite{dpv} the following important property:
\begin{itemize}
 \item[$(d)$] The mappings $h\mapsto T_hg$ are continuous for each $g\in E$.
 \end{itemize}

Throughout the rest of the article $E$ always stands for a translation-invariant Banach space of tempered distributions. We shall call the function $\omega$, given by (\ref{eqw}), the growth function of the translation group of $E$ (in short: \emph{the growth function of $E$}). Note that $ \omega$ is measurable, $\omega(0)=1$, and $\log \omega$ is subadditive.  
We associate to $E$ the Beurling algebra $L^{1}_{\omega}$, i.e., the Banach algebra of measurable functions $u$ such that $||u||_{1,\omega}:=\int_{\mathbb{R}^{n}}|u(x)| \:\omega(x) dx<\infty$. Recall that the dual of $L^{1}_{\omega}$ is $L^{\infty}_{\omega}$, the Banach space of measurable functions satisfying
$
||u||_{\infty,\omega}:=\operatorname*{ess}\sup_{x\in\mathbb{R}^{n}} |u(x)|/\omega(x)<\infty.$ We have proved in \cite{dpv} that 
the convolution $\ast:\mathcal{S}(\mathbb{R}^n)\times\mathcal{S}(\mathbb{R}^n)\rightarrow \mathcal{S}(\mathbb{R}^n)$ extends to $\ast:L^{1}_{\omega}\times E\rightarrow E$ in such a way that $E$ is a Banach module over the Beurling algebra $L^{1}_{\omega}$, i.e., $ ||u\ast g||_{E}\leq ||u||_{1,\omega}||g||_{E}$.

The dual space $E'$ carries two convolution structures which will play a crucial role in the rest of the article. On the one hand, setting $\check{E}=\{g\in \mathcal{S}'(\mathbb{R}^{n}):\: \check{g}\in E\}$, we clearly obtain a well-defined convolution mapping $\ast:E'\times \check{E}\to L^{\infty}_{\omega}$; the following pointwise estimate obviously holds:
\begin{equation}\label{estimateconv1}
|(f\ast g)(x)|\leq \omega(x)\| f\|_{E'}\|\check{g}\|_{E},  \ \ \ \forall x\in\mathbb{R}^{n}. 
\end{equation}
On the other hand, we can associate the Beurling algebra $L^{1}_{\check{\omega}}$ to $E'$ (recall $\check{\omega}(x)=\omega(-x)$) and the convolution of $f\in E'$ and $u\in L^{1}_{\check{\omega}}$ is  defined via transposition:
$
\left\langle u\ast f,g \right\rangle:= \left\langle f,\check {u}\ast g\right\rangle,$ $g\in E.
$ The space $E'$ then becomes a Banach modulo over $L^{1}_{\check{\omega}}$, that is, 
\begin{equation}
\label{estimateconv2}
||u\ast f||_{E'}\leq ||u||_{1,\check{\omega}}||f||_{E'}.
\end{equation}
 
It is important to notice that, in general, $E'$ is not a translation-invariant Banach space of tempered distributions. Indeed, the properties $(a)$ and $(d)$ may fail for $E'$ (e.g., take $E=L^{1}$). We have introduced in \cite{dpv} the space $E'_{\ast}=L^{1}_{\check{\omega}}\ast E'$. Since the Beurling algebra $L^{1}_{\check{\omega}}$ admits bounded approximation unities, it follows from the Cohen-Hewitt factorization theorem \cite{kisynski} that $E'_{\ast}$ is actually a closed linear subspace of $E'$. Thus, $E'_{\ast}$ inheres the Banach modulo structure over $L^{1}_{\check{\omega}}$.  The Banach space of distributions $E'_{\ast}$ possesses the properties $(b)$, $(c)$, and $(d)$; moreover, we have the explicit description \cite[Prop. 5]{dpv} $E'_{\ast}=\{f\in E': \lim_{h\to 0}\|T_{h}f-f\|_{E'}=0\}$ . One can also show \cite[Thm. 2]{dpv} that if $f\in E'_{\ast}$ and $\phi \in\mathcal{S}(\mathbb{R}^{n})$ is such that $\int_{\mathbb{R}^{n}}\phi(x)dx=1$, then 
\begin{equation}
\label{eqapprox}
\lim_{\varepsilon\to0^{+}}\|f-\phi_{\varepsilon}\ast f\|_{E'}=0,
\end{equation}
where $ \phi_{\varepsilon}(x)=\varepsilon^{-n}\phi \left(x/ \varepsilon \right)$. When $E$ is reflexive, we have proved \cite{dpv} that $E'$ is also a translation-invariant Banach space of tempered distributions and in fact $E'=E'_{\ast}$.

\begin{example} \label{ex1}
Typical examples of $E$ are the $L^{p}$-weighted spaces. Let $\eta$ be a polynomially bounded weight, that is, a measurable function $\eta:\mathbb{R}^n\rightarrow (0,\infty)$ that fulfills the requirement $\eta(x+h)\leq M'\eta(x)(1+|h|)^\tau$, for some $M',\tau>0$. We consider the norms
$
||g||_{p,\eta}=\|g \eta\|_{p}=\|g \eta\|_{L^{p}(\mathbb{R}^{n})}$ for $p\in [1,\infty)$ and $||g||_{\infty,\eta}=\|g/\eta\|_{\infty}
$.
Then the space $L^{p}_{\eta}$ consists of those measurable functions such that $||g||_{p,\eta}<\infty$. One clearly has that $E=L^{p}_{\eta}$ are translation-invariant Banach spaces of tempered distributions for $p\in[1,\infty)$. The case $p=\infty$ is an exception, because the properties $(a)$ and $(d)$ fail for $L^{\infty}_{\eta}$. In view of reflexivity, the space $E'_{\ast}$ corresponding to $E=L^{p}_{\eta^{-1}}$ is $E'_{\ast}=E'=L^{q}_{\eta}$ whenever $1<p<\infty$, where $q$ is the conjugate index to $p$. On the other hand, $E'_{\ast}=UC_{\eta}:= \left\{u\in L^{\infty}_{\eta}: \lim_{h\to0}||T_{h}u-u||_{\infty,\eta}=0 \right\}$ for $E=L^{1}_{\eta}$. The weight function of $L^{p}_{\eta}$ is $\omega(h)= \|T_{h}\eta\|_{\infty,\eta}$ for $p\in[1,\infty)$, while that for $L^{\infty}_{\eta}$ is $\check{\omega}$ (cf. \cite[Prop. 10]{dpv}). Another instance of $E$ is the closed subspace of $L^{\infty}_{\eta}$ given by $E=C_{\eta}=\left\{g\in C(\mathbb{R}^{n}): \: \lim_{|x|\to \infty} g(x)/\eta(x)=0\right\};$ in this case, $E'_{\ast}=L^{1}_{\eta}$. 
\end{example}

\subsection{The distribution space $\mathcal{D}'_{E'_{\ast}}$}
\label{the distribution space}

For every nonnegative integer $N$ define the Banach spaces 
$\mathcal{D}_E^N=\{\varphi\in E:\,\|\varphi\|_{E,N}=\max_{|\alpha| \leq N}\|\varphi^{(\alpha)}\|_E<\infty\}$
and the test function space
$$\mathcal{D}_E=\projlim_{N\rightarrow\infty}\mathcal{D}_E^N.$$
We have shown \cite[Prop. 8]{dpv} that $\mathcal{D}_E$ is a Fr\'{e}chet space of smooth functions and actually the following dense and continuous inclusions hold: 
\begin{equation}\label{eqemb1}
\mathcal{S}(\mathbb{R}^n)\hookrightarrow\mathcal{D}_E\hookrightarrow \mathcal{O}_{C}(\mathbb{R}^{n}) \hookrightarrow \mathcal{E}(\mathbb{R}^n),
\end{equation}
where $\mathcal{O}_{C}(\mathbb{R}^{n})$ stands for the test function space corresponding to the space of convolutors $\mathcal{O}_{C}'(\mathbb{R}^{n})$ of $\mathcal{S}'(\mathbb{R}^n)$. When $E$ is reflexive, one can show that $\mathcal{D}_{E}$ is reflexive as well (cf. \cite[Prop. 9]{dpv}).

We define the distribution space $\mathcal{D}'_{E'_{\ast}}$ as the strong dual of $\mathcal{D}_{E}$. 
When $E$ is reflexive, we write $\mathcal{D}'_{E'}=\mathcal{D}'_{E'_{\ast}}$. By (\ref{eqemb1}), we have the (continuous) inclusions:
\begin{equation}\label{eqemb2}
 \mathcal{E}'(\mathbb{R}^{n})\to\mathcal{O}'_{C}(\mathbb{R}^{n})\to\mathcal{D}'_{E'_{\ast}}\to\mathcal{S}'(\mathbb{R}^{n}).
\end{equation}

The notation $\mathcal{D}'_{E'_{\ast}}=(\mathcal{D}_{E})'$ is motivated by the next structural theorem, which in particular tells that every element of $\mathcal{D}'_{E'_{\ast}}$ is the sum of partial derivatives of elements of $E'_{\ast}$. This fact will be extremely important in the subsequent sections, because, contrary to $E'$ in general, the condition $(d)$ from Subsection \ref{the space E} holds for $E'_{\ast}$. See \cite[Thm. 3 and Cor. 4]{dpv} (and their proofs) for a proof of Theorem \ref{th structure}.

\begin{theorem}[\cite{dpv}]\label{th structure}
Let $B'\subset \mathcal{S}'(\mathbb{R}^n)$. The following statements are equivalent:
\begin{itemize}
\item [$(i)$] $B'$ is a bounded subset of $\mathcal{D}'_{E'_{\ast}}$.

\item [$(ii)$] $\psi \ast B'=\{\psi \ast f:\: f\in B'\}$ is a bounded subset of $E'_{\ast}$ for each $\psi\in \mathcal{S}(\mathbb{R}^n)$.

\item [$(iii)$] There are $M>0$ and $N\in\mathbb{N}$ such that every $f\in B'$ admits a representation
\begin{equation}\label{eq:representation}
 f=\sum_{|\alpha|\leq N}f_{\alpha}^{(\alpha)}
\end{equation}
 with continuous functions $f_{\alpha}\in E'_{\ast}\cap UC_{\omega}\subset E'\cap L^{\infty}_{\omega}$ satisfying the uniform bounds $||f_{\alpha}||_{E'}<M$ and $||f_{\alpha}||_{\infty,\omega}<M$. 
\end{itemize}
Furthermore, the functions in the representation (\ref{eq:representation}) can be chosen to have the form
$f_{\alpha}= f\ast \check{\varrho}_{\alpha}$, where the $\varrho_{\alpha}\in E$ are continuous functions of compact support that depend only on the set $B'$. 
\end{theorem}
We shall often apply Theorem \ref{th structure} to the case of singleton sets $B'=\{f\}$. Part $(ii)$ from Theorem \ref{th structure} also suggests to embed the distribution space $\mathcal{D}_{E'_{\ast}}'$ into the space of $E'_{\ast}$-valued tempered distributions as follows. Define the injection
\begin{equation}
\label{embedding}
\iota:\mathcal{D}_{E'_{\ast}}'\to\mathcal{S}'(\mathbb{R}^{n},E'_{\ast}),
\end{equation}
where $\iota(f)=\mathbf{f}$ is given by 
\begin{equation}
\label{eq:8}
\left\langle\mathbf{f},\varphi\right\rangle=f\ast \check{\varphi}, \ \ \ \varphi\in\mathcal{S}(\mathbb{R}^{n}).
\end{equation}
One can prove \cite[Sect. 4]{dpv} that (\ref{embedding}) is continuous, has closed range given by the subspace of $E'_{\ast}$-valued distributions that commute with every translation operator, namely \cite[Cor. 3]{dpv}, 
\begin{equation}
\label{eqrange}
\iota(\mathcal{D}_{E'_{\ast}}')=\{\mathbf{f}\in\mathcal{S}'(\mathbb{R}^{n},E'_{\ast}):\: \left\langle T_{h}\mathbf{f},\varphi\right\rangle = T_{h} \left\langle \mathbf{f},\varphi\right\rangle, \:\forall h\in\mathbb{R}^{n},\:\forall\varphi\in \mathcal{S}(\mathbb{R}^{n})\},
\end{equation}
and its inverse mapping is sequentially continuous. We collect the latter fact in the first part of the following theorem (see \cite[Cor. 5]{dpv} for a proof).

\begin{theorem}[\cite{dpv}]\label{sequences} A sequence $\left\{f_{j}\right\}_{j=0}^{\infty}$ $($or similarly, a filter with a countable or bounded basis$)$ is (strongly) convergent in $\mathcal{D}_{E'_{\ast}}'$ if and only if $\left\{f_{j}\ast \psi\right\}_{j=0}^{\infty}$ is convergent in $E'$ for all $\psi\in\mathcal{S}(\mathbb{R}^{n})$. In addition, these statements are also equivalent to the existence of $N\in\mathbb{N}$ and continuous functions of compact support $\varrho_{\alpha}\in E$, $|\alpha|\leq N$, such that  $f_j=\sum_{|\alpha|\leq N} f^{(\alpha)}_{\alpha,j}$ and the sequences $\left\{f_{\alpha,j}\right\}_{j=0}^{\infty}$ are convergent in both $E'_{\ast}$ and $L^{\infty}_{\omega}$, where $f_{\alpha,j}=f_{j}\ast \check{\varrho}_{\alpha}\in E'_{\ast}\cap UC_{\omega}$.
\end{theorem}

We conclude this section with some examples of spaces $\mathcal{D}'_{E'_{\ast}}$.
\begin{example}\label{ex2} The Schwartz spaces $\mathcal{D}'_{L^{p}}$ and $\mathcal{B}'$ are particular instances of $\mathcal{D}'_{E'_{\ast}}$. More generally, retaining the notation from Example \ref{ex1}, the choices $E=L^{q}_{\eta^{-1}}$ lead to the spaces $\mathcal{D}'_{L^{p}_{\eta}}$, $1<p<\infty$. When $p=1$, we use the notation $\mathcal{B}'_{\eta}:=(\mathcal{D}_{L^{1}_{\eta}})'$. We set $\dot{\mathcal{B}}_{\eta}:=\mathcal{D}_{C_{\eta}}$ so that $\mathcal{D}_{L_{\eta}^1}'=(\mathcal{D}_{C_{\eta}})'=(\dot{\mathcal{B}}_\eta)'$. We denote as  $\dot{\mathcal{B}}'_{\eta}$ the closure of $\mathcal{D}(\mathbb{R}^{n})$ in $\mathcal{B}'_{\eta}$.
\end{example}

\begin{remark}
\label{remarkemb1} The embeddings (\ref{eqemb1}) and (\ref{eqemb2}) can be refined to $\mathcal{D}_{L^{1}_{\omega}}\hookrightarrow\mathcal{D}_{E}\hookrightarrow\dot{\mathcal{B}}_{\check{\omega}}$ and hence the continuous inclusions $\mathcal{D}'_{L^{1}_{\check{\omega}}}\rightarrow\mathcal{D}'_{E'_{\ast}}\rightarrow{\mathcal{B}}'_{\omega}$ (cf. \cite[Thm. 4 ]{dpv}). When $E$ is reflexive $\mathcal{D}'_{L^{1}_{\check{\omega}}}\hookrightarrow\mathcal{D}'_{E'}\hookrightarrow\dot{\mathcal{B}}'_{\omega}$.
\end{remark}
\section{Boundary values of holomorphic functions} \label{boundary values}
In this section we study boundary values in the context of the space $\mathcal{D}'_{E'_{\ast}}$. We shall characterize those holomorphic functions on tube domains, whose bases are open convex cones, that have boundary values in the strong topology of $\mathcal{D}'_{E'_{\ast}}$. Our first goal is to obtain such characterization for holomorphic functions defined on a truncated wedge. We begin with a useful lemma.
The constants $M'$ and $\tau$ are those occurring in the estimate (\ref{eqw}) for the weight function of the translation-invariant Banach space of tempered distributions $E$.

\begin{lemma}\label{bvl1} Let $V\subsetneq\mathbb{R}^{n}$ be an open set and let $F$ be holomorphic on the tube $T^{V}$. Suppose that $F(\:\cdot \:+ iy)\in E'$ for $y\in V$ and
\begin{equation}
\label{bveq01}
\sup_{y\in V}\frac{(d_{V}(y))^{\kappa_{1}}}{(1+d_{V}(y))^{\kappa_{2}}}||F(\:\cdot \:+ iy)||_{E'}=M<\infty \ \ \ \ (\kappa_{1},\kappa_{2}\geq 0) .
\end{equation}
Then, for every $\alpha\in\mathbb{N}^{n}$ one has $F^{(\alpha)}(\:\cdot \:+ iy)\in E'$ for all $y\in V$ and
\begin{equation}
\label{bveq02}
\sup_{y\in V}\frac{(d_{V}(y))^{\kappa_{1}+|\alpha|}}{(1+d_{V}(y))^{\kappa_{2}+\tau}}||F^{(\alpha)}(\:\cdot \:+ iy)||_{E'}\leq (2\pi)^{n/2}M M'\frac{(1+\lambda)^{\kappa_{2}}}{(1-\lambda)^{\kappa_{1}}}\left(\frac{\sqrt{n}}{\lambda}\right)^{|\alpha|}\alpha!, \ \ \ \lambda\in(0,1).
\end{equation}
Furthermore, the $E'$-valued mapping $\mathbf{F}:T^{V}\to E'$ is holomorphic, where
\begin{equation}
\label{bveq03}
\mathbf{F}(x+iy)=T_{x}(F(\:\cdot\:+y)).
\end{equation}
\end{lemma}
\begin{proof} The assumption $V\neq \mathbb{R}^{n}$ is only used to ensure that $d_{V}(y)<\infty$ for all $y\in V$. Fix $0<\lambda<1$. Let $\varphi\in\mathcal{D}(\mathbb{R}^{n})$. Let $\zeta=u+iv=(\zeta_1,\zeta_2,\dots,\zeta_n)$ be an arbitrary point in the distinguished boundary of the polydisc $\mathbb{D}^{n}$, that is, $|\zeta_{1}|=|\zeta_2|=\dots=|\zeta_n|=1$. We write $s=t+i\sigma\in \mathbb{C}$. For arbitrary $y\in V$, define the function $G(s)=G_{y,\zeta}(s)=\int_{\mathbb{R}^{n}}F(x+iy+s\zeta)\varphi(x)dx = \left\langle T_{tu-\sigma v}F(\:\cdot \:+ i(y+tv+\sigma u )), \varphi  \right\rangle$. It is clear that $G$ is defined and holomorphic in the disc $\{s\in\mathbb{C}: |s|<d_{V}(y)/\sqrt{n}\}.$  Note that
\begin{align*}
|G(s)|&\leq \frac{M\omega(tu-\sigma v) (1+d_{V}(y+tv+\sigma u))^{\kappa_{2}}}{(d_{V}(y+tv+\sigma u))^{\kappa_{1}}}\:\|\varphi\|_{E}
\\
&
\leq \frac{(1+\lambda)^{\kappa_{2}}MM'(1+d_{V}(y))^{\tau+\kappa_2}}{((1-\lambda)d_{V}(y))^{\kappa_1}} \:\|\varphi\|_{E} \ \ \   \mbox{for } |s|\leq\frac{\lambda d_{V}(y)}{\sqrt{n}}.
\end{align*}
The Cauchy inequality for derivatives applied to circle $|s|=(\lambda/\sqrt{n}) d_{V}(y)$ thus yields
$$
|G^{(N)}(0)|\leq \frac{n^{N/2}(1+\lambda)^{\kappa_{2}}MM'(1+|d_{V}(y)|)^{\tau+\kappa_2}}{\lambda^{N}(1-\lambda)^{\kappa_{1}}(d_{V}(y))^{\kappa_1+N}} \: N!\|\varphi\|_{E}, \ \ \ N=0,1,2,\dots,
$$
i.e.,
$$
\left|P_{N}(\zeta)\right|\leq \frac{n^{N/2}(1+\lambda)^{\kappa_{2}}MM'(1+|d_{V}(y)|)^{\tau+\kappa_2}}{\lambda^{N}(1-\lambda)^{\kappa_{1}}(d_{V}(y))^{\kappa_1+N}}\:\|\varphi\|_{E}, \ \ \ N=0,1,2,\dots,
$$
where $P_{N}(\zeta)=\sum_{|\alpha|=N}\zeta^{\alpha}\left\langle F^{(\alpha)}(\:\cdot \:+ iy),\varphi\right\rangle/\alpha!$. Integrating $|P_{N}(\zeta)|^{2}$ over $(\partial\mathbb{D})^{n}$, we obtain
$$
\left|\left\langle F^{(\alpha)}(\:\cdot \:+ iy),\varphi\right\rangle\right|\leq \frac{(1+\lambda)^{\kappa_{2}}MM' n^{|\alpha|/2} \alpha!(2\pi)^{n/2}(1+|d_{V}(y)|)^{\tau+\kappa_2}}{\lambda^{|\alpha|}(1-\lambda)^{\kappa_{1}}(d_{V}(y))^{\kappa_1+|\alpha|}}\:\|\varphi\|_{E},
$$
for all $\varphi\in\mathcal{D}(\mathbb{R}^{n})$, $y\in V$, and $\alpha\in\mathbb{N}^{n}$. The very last inequality is equivalent to (\ref{bveq02}). To show that (\ref{bveq03}) is holomorphic, it is enough to fix $z\in V$ and $\zeta \in (\partial \mathbb{D})^{n}$ and to verify that $\mathbf{F}(z+s\zeta)$ is holomorphic in $|s|< d_{V}(y)/\sqrt{n}$. By the previous argument, $F(\:\cdot\:+iy+s\zeta)=\sum_{k=0}^{\infty}s^{k}g_{k}$, with $g_{k}=\sum_{|\alpha|=k} \zeta^{\alpha} F^{(\alpha)}(\:\cdot \:+ iy)/\alpha!$, is a convergent power series in $E'$ for $|s|< d_{V}(y)/\sqrt{n}$. Employing the continuity of $T_{x}$, we obtain $\mathbf{F}(z+s\zeta)=\sum_{k=0}^{\infty}s^{k}T_{x}g_{k}$. Hence, $\mathbf{F}$ is holomorphic.
\end{proof}

Lemma \ref{bvl1} has the ensuing consequence.
\begin{corollary}
\label{bvc1} Let $V\subseteq \mathbb{R}^{n}$ and let $F$ be holomorphic in $T^{V}$ such that $F(\:\cdot\:+iy)\in E'$ for all $y\in V$ and $\sup_{y\in K}\|F(\:\cdot\:+iy)\|_{E'}<\infty$ for every compact subset $K\subset V$. Then $\lim_{y\to y_{0}}\|F(\:\cdot\:+iy)-F(\:\cdot\:+iy_0)\|_{E'}=0$ for each $y_{0}\in V$.
\end{corollary}
\begin{proof} The statement is local, so we may assume $V\neq \mathbb{R}^{n}$. The mapping (\ref{bveq03}) is continuous at $z_{0}=iy_{0}$ and $F(\:\cdot\:+iy)=\mathbf{F}(iy)$.
\end{proof}

In the rest of the section we mainly focus our attention on tubes whose bases are either truncated convex cones or full convex cones. We now show our first main result:

\begin{theorem}\label{thbvw}
Let $C$ be an open convex cone and let $r>0$. Suppose that $F$ is holomorphic on the tube $T^{C(r)}$ and satisfies
\begin{equation}
\label{bveq04} F(\: \cdot\:+iy)\in \mathcal{D}'_{E'_{\ast}}\:, \  \  \ \mbox{for every }y\in C(r),
\end{equation}
and the sets $\{F(\: \cdot\:+iy): \: r'<|y|<r,\: y\in C\}$ are bounded in $\mathcal{D}'_{E'_{\ast}}$ for each $r'>0$.
Then, the following three statements are equivalent:

\begin{itemize}
\item [$(i)$] $F$ satisfies
\begin{equation}
\label{bveq05} F(\: \cdot\:+iy)\in E', \  \  \ y\in C(r),
\end{equation}
and the bound
\begin{equation}
\label{bveq06} \|F(\: \cdot\:+iy)\|_{E'}\leq \frac{M}{(d_{C(r)}(y))^{\kappa}}, \  \  \ y\in  C(r).
\end{equation}
\item [$(ii)$] $F$ has boundary values in $\mathcal{D}'_{E'_{\ast}}$, namely, there is $f\in\mathcal{D}'_{E'_{\ast}}$ such that
\begin{equation}
\label{bveq07}
f= \underset{y\in C}{\lim_{y\to 0}}\: F(\:\cdot\:+iy)\  \  \  \mbox{strongly in }\mathcal{D}'_{E'_{\ast}}.
\end{equation}
\item [$(iii)$] The set $\{F(\: \cdot\:+iy): \: y\in C(r)\}$ is bounded in $\mathcal{D}'_{E'_{\ast}}$.
\end{itemize}
In addition, if any of these equivalent conditions is satisfied, then $F(\:\cdot\:+ y)\in E'_{\ast}$ for every $y\in C(r)$.
\end{theorem}
\begin{proof}
The implication $(ii)\Rightarrow(iii)$ is obvious.

$(i)\Rightarrow(ii)$. Assume (\ref{bveq05}) and (\ref{bveq06}). If $C=\mathbb{R}^{n}$, the result follows from Corollary \ref{bvc1}. Suppose then that $C\neq \mathbb{R}^{n}$ (i.e., $0\notin C$).  Thanks to Theorem \ref{th structure}, we can write
\begin{equation}
\label{bveq08}
F(z)=\sum_{|\alpha|\leq N} \partial ^{\alpha}_z F_{\alpha}(z), \  \  \  z\in T^{C(r)},
\end{equation}
where each $F_{\alpha}$ has the form
\begin{equation}
\label{bveq09}
F_{\alpha}(z)= (F(\:\cdot\:+iy)\ast \check{\varrho}_{\alpha})(x)=\int_{\operatorname*{supp}\varrho_{\alpha}}F_{\alpha}(z+\xi) \varrho_{\alpha}(\xi)d\xi \ \ \ (z=x+iy)
\end{equation} and each $\varrho_{\alpha}\in E$ is a continuous function of compact support. Thus, each $F_{\alpha}$ is also holomorphic on the tube $T^{C(r)}$, satisfies $F_{\alpha}(\:\cdot\:+iy)\in E'_{\ast}$ for every $y\in C(r)$ (because $\check{\varrho}_{\alpha}\in L^{1}_{\omega}$),
the $E'$-norm estimate (cf. (\ref{estimateconv2}))
\begin{equation}
\label{bveq010}
\|F_{\alpha}(\:\cdot\:+iy)\|_{E'}\leq \frac{M\|\varrho_{\alpha}\|_{1,\omega}}{(d_{C(r)}(y))^{\kappa}}, \ \ \ y\in C(r),
\end{equation}
and the pointwise estimate (cf. (\ref{estimateconv1}))
\begin{equation}
\label{bveq011}
|F_{\alpha}(x+iy)|\leq \frac{M \|\varrho_{\alpha}\|_{E}\:\omega(x)}{(d_{C(r)}(y))^{\kappa}}, \ \ \ x+iy\in T^{C(r)}.
\end{equation}
Making use of by Corollary \ref{bvc1}, the mappings $y\in C(r)\mapsto F_{\alpha}(\:\cdot\: +i y)\in E'_{\ast}$ are continuous. The pointwise estimate (\ref{bveq011}) implies that each $F_{\alpha}$ has boundary values in $\mathcal{S}'(\mathbb{R}^{n})$ \cite{C-M,Vcomplex}. Set
\begin{equation}
\label{bveq012}
f_{\alpha}=\underset{y\in C}{\lim_{y\to 0}}F_{\alpha}(\:\cdot\: +i y) \ \ \  \mbox{in } \mathcal{S}'(\mathbb{R}^{n}), \ \ |\alpha|\leq N.
\end{equation}
In view of (\ref{bveq08}), it suffices to show that each $f_{\alpha}\in\mathcal{D}'_{E'_{\ast}}$ and that the limit (\ref{bveq012}) actually holds in $\mathcal{D}'_{E'_{\ast}}$. We may assume that $\kappa\in\mathbb{N}$. Let $\psi\in\mathcal{S}(\mathbb{R}^{n})$ and write $\Psi(x,y)=\sum_{|\beta|\leq \kappa} \psi^{(\beta)}(x)(iy)^{\beta}/\beta!$. Pick $\theta\in C(r/4)$. Since $-\theta\notin C$, we can find $M_{1}$ such that $\lambda\leq M_{1}d_{C}(y+\lambda \theta)$ for every $y\in C$ and $\lambda>0$. In particular, $\lambda\leq M_{1}d_{C(r)}(y+\lambda \theta)$ for $\lambda\in (0,1)$ and $y\in C(r/4)$.  Applying the Stokes theorem as in \cite[p. 67]{hormander}, we can write
$$
f_{\alpha}\ast\psi= \Psi(\:\cdot\:,\theta)\ast F_{\alpha}(\cdot\: +i\theta)+\sum_{|\beta|= \kappa+1}\frac{(i\theta)^{\beta}(\kappa+1)}{\beta!}\int_{0}^{1}\lambda^{\kappa}(F_{\alpha}(\:\cdot\:+i\lambda\theta)\ast\psi^{(\beta)})\:d\lambda
$$
and, for $y\in C(r/4)$,
$$
F_{\alpha}(\: \cdot \: +iy)\ast\psi= \Psi(\:\cdot\:,\theta)\ast F_{\alpha}(\cdot\: +i\theta+iy)+\sum_{|\beta|= \kappa+1}\frac{(i\theta)^{\beta}(\kappa+1)}{\beta!}\int_{0}^{1}\lambda^{\kappa}(F_{\alpha}(\:\cdot\:+i\lambda\theta+iy)\ast\psi^{(\beta)})\:d\lambda,
$$
where the integrals are interpreted as $E'_{\ast}$-valued integrals in the Bochner sense. By Theorem \ref{sequences} and Corollary \ref{bvc1}, the net $\Psi(\:\cdot\:,\theta)\ast F_{\alpha}(\cdot\: +i\theta+iy)\to \Psi(\:\cdot\:,\theta)\ast F_{\alpha}(\cdot\: +i\theta)$ in $E'_{\ast}$. Furthermore, using the estimate (\ref{bveq010}), we majorize $\lambda^{\kappa}\|F_{\alpha}(\:\cdot\:+i\lambda\theta+iy)\ast\psi^{(\beta)}\|_{E'}\leq (M_{1})^{\kappa}M \|\psi^{(\beta)}\|_{1,\omega}\|\varrho_{\alpha}\|_{1,\omega}$ and the dominated convergence theorem for Bochner integrals thus yields
$$
f_{\alpha}\ast\psi=
\underset{y\in C}{\lim_{y\to0} }(F_{\alpha}(\cdot\: +iy)\ast\psi) \  \  \ \mbox{in } E'_{\ast}.
$$
Since this holds for every $\psi\in \mathcal{S}(\mathbb{R}^{n})$, Theorem \ref{sequences} implies
$$
f_{\alpha}=
\underset{y\in C}{\lim_{y\to0} }F_{\alpha}(\cdot\: +iy) \  \  \ \mbox{strongly in } \mathcal{D}'_{E'_{\ast}}
$$
and (\ref{bveq07}) follows at once.

$(iii)\Rightarrow(i)$. Using part $(iii)$ of Theorem \ref{th structure}, we can write $F$ as in (\ref{bveq08}) where each $F_{\alpha}$ is holomorphic in $T^{C(r)}$, $F_{\alpha}(\:\cdot\:+iy)\in E'_{\ast}$ and $\sup_{y\in C(r)}\|F_{\alpha}(\:\cdot\:+iy)\|_{E'}<\infty$. The assertion $(i)$ is a consequence of Lemma \ref{bvl1}. In addition, we get that the holomorphic function (\ref{bveq03}) actually takes values in $E'_{\ast}$. Thus $\mathbf{F}^{(\alpha)}_{\alpha}(z)\in E'_{\ast}$ for all $z\in T^{C(r)}$, whence $\mathbf{F}:T^{C(r)}\to E'_{\ast}$.
\end{proof}

\begin{corollary}
\label{bvc2} Let $V\subseteq \mathbb{R}^{n}$ be an open set and let $F$ be holomorphic in $T^{V}$. If $F(\:\cdot\:+iy)\in \mathcal{D}'_{E'_{\ast}}$ for all $y\in V$ and $\{F(\:\cdot\:+iy): \: y\in K\}$ is bounded in $\mathcal{D}'_{E'_{\ast}}$ for every compact subset $K\subset V$, then actually $F(\:\cdot\:+iy)\in E'_{\ast}$ for all $y\in V$, $\sup_{y\in K}\|F(\:\cdot\:+iy)\|_{E'}<\infty$ for every compact $K\subset V$, and the $E'_{\ast}$-valued function (\ref{bveq03}) is holomorphic in $T^{V}$.  If in addition $V\neq \mathbb{R}^{n}$ and the set $\{F(\:\cdot\:+iy): \: y\in V\}$ is bounded in $\mathcal{D}'_{E'_{\ast}}$, then there is $\kappa\geq 0$ such that $\sup_{y\in V}(d_{V}(y))^{\kappa}(1+d_{V}(y))^{-\tau}\|F(\:\cdot\:+iy)\|_{E'}<\infty$.
\end{corollary}
\begin{proof} The first part of the corollary follows from the second one. Exactly the same argument from the proof of the implication $(iii)\Rightarrow(i)$ of Theorem \ref{thbvw} shows the second assertion.
\end{proof}

Using Theorem \ref{thbvw}, we can derive the following result.
\begin{corollary}
\label{bvc3} Let $X\subset \mathcal{S}'(\mathbb{R}^{n})$ be a Banach space. Assume that the inclusion mapping $X\to \mathcal{S}'(\mathbb{R}^{n})$ is continuous. Let $C$ be an open convex cone and $r>0$. If $F$ is holomorphic on the tube $T^{C(r)}$ and satisfies
$$
 F(\: \cdot\:+iy)\in X \ \mbox{ and } \   \|F(\: \cdot\:+iy)\|_{X}\leq \frac{M}{(d_{C(r)}(y))^{\kappa}}, \  \  \ y\in C(r),
$$
then $\displaystyle\underset{y\in C}{\lim_{y \to0}} F(\: \cdot\:+iy)$ exists in $\mathcal{S}'(\mathbb{R}^{n})$.
\end{corollary}
\begin{proof}
Let $\mathcal{S}_{j}(\mathbb{R}^{n})$ be the completion of 
$\mathcal{S}(\mathbb{R}^{n})$ in the norm $$q_j(\varphi)=\underset{|\alpha|\leq j}{\sup_{x\in \mathbb{R}^n}}(1+|x|)^j|\varphi^{(\alpha)}(x)|, \ \ \ j\in\mathbb{N}.
$$ Notice that each $\mathcal{S}_{j}(\mathbb{R}^{n})$ is a translation-invariant Banach spaces of tempered distributions. The embeddings $\mathcal{S}_{j+1}(\mathbb{R}^{n})\hookrightarrow\mathcal{S}_{j}(\mathbb{R}^{n})$  are compact, $\mathcal{S}(\mathbb{R}^{n})=\operatorname*{proj} \lim_{j\in \mathbb{N}}\mathcal{S}_{j}(\mathbb{R}^{n})$, and hence $\mathcal{S}'(\mathbb{R}^{n})=\operatorname*{ind}\lim_{j\in\mathbb{N}}\mathcal{S}_{j}'(\mathbb{R}^{n})$ is a regular inductive limit of Banach spaces. Thus, there are $M_1>0$ and $j_{0}\in\mathbb{N}$ such that $\|f\|_{\mathcal{S}_{j_{0}}'(\mathbb{R}^{n})}\leq M_{1}\|f\|_{X}$, for all $f\in X$. The assertion then follows by applying Theorem \ref{thbvw} with $E'=\mathcal{S}_{j_{0}}'(\mathbb{R}^{n})$.
\end{proof}

Observe that Corollary \ref{bvc3} provides sufficient conditions for the existence of boundary values in $\mathcal{S}'(\mathbb{R}^{n})$ in terms of rather general norms; however, in contrast with Theorem \ref{thbvw}, very little can be said about the boundary distribution $f=\lim_{y\in C\to 0} F(\: \cdot\:+iy)$ unless the Banach space $X$ possesses a richer structure. It should also be noticed that, as well-known, the holomorphic function $F$ is uniquely determined by its distributional boundary values $f$.

 We now turn our attention to holomorphic functions satisfying global estimates over a tube having a full open convex cone as base. We need to introduce some notation in order to move further.
Let $C\subset \mathbb{R}^{n}$ be an open convex cone. Set $\mathcal{S}'(C^\ast+\overline{B}(a))=\{g\in\mathcal{S}'(\mathbb{R}^{n}):\:\operatorname*{supp}g \subseteq C^{\ast}+ \overline{B}(a)\}$. The Laplace transform of $g\in\mathcal{S}'(C^\ast+\overline{B}(a))$ is defined \cite{V} as the holomorphic function
$$
\mathcal{L}\left\{g;z\right\}=\left\langle g(\xi),e^{iz\cdot \xi}\right\rangle, \  \  \ z\in T^{C}.
$$
The above distributional evaluation is well-defined because  $\mathcal{S}'(C^\ast+\overline{B}(a))$ is canonically isomorphic to the dual of the function space $\mathcal{S}(C^\ast+\overline{B}(a))$ (cf. \cite{Vcomplex,vladimirov-d-z}).

We are interested in the class of holomorphic functions $F:T^{C}\to \mathbb{C}$ that satisfy the following two conditions:
\begin{equation}
\label{bveq013} F(\: \cdot\:+iy)\in E', \  \  \ \mbox{for all }y\in C,
\end{equation}
and the estimate (for some constants $M$, $m$, and $k$)
\begin{equation}
\label{bveq014} \|F(\: \cdot\:+iy)\|_{E'}\leq M (1+|y|)^{m}e^{a|y|}\left(1+\frac{1}{d_{C}(y)}\right)^{k}, \  \  \ y\in C.
\end{equation}
Because of Corollary \ref{bvc2}, the membership relation (\ref{bveq013}) is equivalent to $F(\: \cdot \:+iy)\in E'_{\ast}$.

We now show that these holomorphic functions are in one-to-one correspondence with those elements of $\mathcal{D}'_{E'_{\ast}}$ having Fourier transforms with supports in the set $C^\ast+\overline{B}(a)$. We work with the constants in the Fourier transform as
$$\hat{\varphi}(\xi)=\int_{\mathbb{R}^n}e^{-i\xi\cdot x} \varphi(x)dx, \ \ \ \varphi\in\mathcal{S}(\mathbb{R}^{n}).$$
The next theorem extends various results by Carmichael \cite{carmichael83,carmichael87} and Vladimirov \cite{V} (obtained by them in the particular cases when $E'=L^{p}$ or when $E'$ is an $L^{2}$ based Sobolev space).
\begin{theorem}\label{analyticc} Let $C\subset \mathbb{R}^{n}$ be an open convex cone and let $a\geq 0$. If
$f\in \mathcal{D}_{E'_{\ast}}'$ is such that $\hat{f}\in\mathcal{S}'(C^\ast+\overline{B}(a))$, then the holomorphic function
\begin{equation}
\label{bveq015} F(z)=(2\pi)^{-n}\mathcal{L}\{\hat{f};z\},  \  \  \  z\in T^{C},
\end{equation}
satisfies (\ref{bveq013}), (\ref{bveq014}) and (\ref{bveq07}).

Conversely, if $F$ is a holomorphic function on $T^{C}$ that satisfies the condition (\ref{bveq013}) and for every subcone $C'\Subset C$ and $\varepsilon>0$ there are $M=M(C',\varepsilon), \kappa=\kappa(C',\varepsilon)>0$ such that
\begin{equation}
\label{bveq016} \|F(\: \cdot\:+iy)\|_{E'}\leq M\frac{e^{(a+\varepsilon)|y|}}{|y|^{\kappa}}, \  \  \ y\in C',
\end{equation}
then there is $f\in \mathcal{D}'_{E'_{\ast}}$ with $ \operatorname*{supp} \hat{f}\subseteq C^\ast+\overline{B}(a)$ such that (\ref{bveq015}) holds.

\end{theorem}
\begin{proof}
Assume that $f\in \mathcal{D}_{E'_{\ast}}'$ is such that $\operatorname*{supp}\hat{f}\subseteq C^\ast+\overline{B}(a)$. Set $\mathbf{f}=\iota(f)\in \mathcal{S}'(\mathbb{R}^{n},E'_{\ast})$ (cf. (\ref{eq:8})). Then,
$\hat{\mathbf{f}}\in\mathcal{S}'(C^\ast+\overline{B}(a),E'_{\ast})=\left\{\mathbf{g}\in\mathcal{S}'(\mathbb{R}^{n},E'_{\ast}):\:\operatorname*{supp} \mathbf{g}\subseteq C^\ast+\overline{B}(a)\right\}.$ The same procedure used to identify $\mathcal{S}'(C^\ast+\overline{B}(a))$ with the dual of $\mathcal{S}(C^\ast+\overline{B}(a))$ \cite{vladimirov-d-z} shows that $\mathcal{S}'(C^\ast+\overline{B}(a),E'_{\ast})$ is canonically isomorphic to $L_{b}(\mathcal{S}(C^\ast+\overline{B}(a)),E'_{\ast})$. So we identify the latter two spaces. This allows us to define the Laplace transform of the $E'_{\ast}$-valued distribution $(2\pi)^{-n}\hat{\mathbf{f}}\in\mathcal{S}'(C^\ast+\overline{B}(a),E'_{\ast})$ as
$\mathbf{F}(z):=(2\pi)^{-n}\mathcal{L}\{\hat{\mathbf{f}};z\}=(2\pi)^{-n}\langle \hat{\mathbf{f}},e^{iz\cdot \xi}\rangle\in E'_{\ast},$ $z\in T^{C}.$
Clearly, $\mathbf{F}$ is holomorphic in $z\in T^{C}$ with values in $E'_{\ast}$ and $\mathbf{F}(z)\to\mathbf{f}$ as $z\in C\to 0$ in $\mathcal{S}'(\mathbb{R}^{n},E'_{\ast})$. It is easy to see that $\mathbf{F}(x+iy)=T_{x}F(\:\cdot\:+iy)\in E'_{\ast}$ and we obtain at once (\ref{bveq013}) by setting $x=0$. Furthermore, $\iota(F(\:\cdot\:+iy))=\mathbf{F}(\:\cdot\:+iy)\to\mathbf{f}=\iota(f)$ in $E'_{\ast}$; hence, Theorem \ref{sequences} yields the limit relation (\ref{bveq07}). Next, one readily sees that $\mathbf{F}(z)$ satisfies the estimate
\begin{equation}
\label{bveq017}
\left\|\mathbf{F}(z)\right\|_{E'}\leq M (1+|z|)^{m}e^{a|\Im m\:z|}\left(1+\frac{1}{d_{C}\left(\Im m\:z\right)}\right)^{k},  \  \  \  z\in T^{C},
\end{equation}
for some constants $m,k,M>0$. The bound (\ref{bveq014}) follows by setting $z=iy$ in (\ref{bveq017}). The proof of (\ref{bveq017}) is exactly the same as in the scalar-valued case. We give it for the sake of completeness. Since $\hat{\mathbf{f}}:\mathcal{S}(C^\ast+\overline{B}(a))\to E'_{\ast}$ is continuous, there are constants $k\in\mathbb{N}$ and $M_1>0$ such that
$$
(2\pi)^{-n}\|\langle \hat{\mathbf{f}},\phi\rangle\|_{E'}\leq M_1\underset{\:\xi\in C^\ast+\overline{B}(a)}{\sup_{\: 0\leq |\alpha|\leq k}}(1+|\xi|)^{k}\left|\phi^{(\alpha)}(\xi)\right|, \  \  \  \forall\phi\in \mathcal{S}(C^\ast+\overline{B}(a)).
$$
Setting $\phi(\xi)=e^{i z\cdot\xi}$, $z=x+iy\in T^{C}$, in the above inequality, we obtain
\begin{align*}
\left\|\mathbf{F}(z)\right\|_{E'}&\leq M_1(1+|z|)^{k} \underset{|\xi_{2}|\leq a}{\sup_{\xi_{1}\in C^{\ast}}}(1+|\xi_{1}+\xi_{2}|)^{k}e^{-y\cdot\xi_{1}}e^{-y\cdot \xi_{2}}
\\
&
\leq (a+1)^{k}M_1(1+|z|)^{k}e^{a|y|} \sup_{\xi\in C^{\ast}}(1+|\xi|)^{k}e^{-|\xi|d_{C}(y)}
\\
&
\leq M e^{a|y|}(1+|z|)^{k}\left(1+\frac{1}{d_{C}\left(y\right)}\right)^{k},
\end{align*}
which gives (\ref{bveq017}) with $M=(a+1)^{k}M_1\sup_{\xi\in C^{\ast}}(1+|\xi|)^{k}e^{-|\xi|}$ and $m=k$.

Conversely, assume (\ref{bveq013}) and (\ref{bveq016}). As in the proof of Theorem \ref{thbvw}, we express $F$ as in (\ref{bveq08}), where each $F_{\alpha}$ is holomorphic in $T^{C}$ and satisfies: $F_{\alpha}(\:\cdot\:+iy)\in E'_{\ast}$ for $y\in C$ and the estimates
$$
\|F_{\alpha}(\:\cdot\:+iy)\|_{E'}\leq M_{\alpha}\frac{e^{(a+\varepsilon)|y|}}{|y|^{\kappa}} \ \mbox{ and } |F_{\alpha}(x+iy)|\leq M_{\alpha}\frac{\omega(x)e^{(a+\varepsilon)|y|}}{|y|^{\kappa}}, \ \ \ x+iy\in T^{C'},
$$
where the constants $M_{\alpha}$ and $\kappa$ are only dependent on the subcone $C'\Subset C$ and $\varepsilon$. The pointwise estimate and Vladimirov's theorem \cite[p. 167]{V} imply that there are $f_{\alpha}\in \mathcal{S}'(\mathbb{R}^{n})$ with $\operatorname*{supp}\hat{f}_{\alpha}\subseteq C^\ast+\overline{B}(a)$ such that $F_{\alpha}(z)=(2\pi)^{-n}\mathcal{L}\{\hat{f}_{\alpha};z\}$. Theorem \ref{thbvw} gives $f_{\alpha}\in \mathcal{D}'_{E'_{\ast}}$. Hence, (\ref{bveq015}) holds with $f=\sum_{|\alpha|\leq N}f^{(\alpha)}_{\alpha}$. This completes the proof.
\end{proof}

Theorem \ref{analyticc} leads to the following general criterion for concluding that a holomorphic function is the Laplace transform of a tempered distribution. The proof goes in the same lines as that of Corollary \ref{bvc3} and we therefore omit it.
\begin{corollary}
\label{bvc5} Let $X\subset \mathcal{S}'(\mathbb{R}^{n})$ be a Banach space for which the inclusion mapping $X\to \mathcal{S}'(\mathbb{R}^{n})$ is continuous and let $C$ be an open convex cone. If $F$ is holomorphic on the tube $T^{C}$ and satisfies
\begin{equation}\label{bveq018}
 F(\: \cdot\:+iy)\in X \ \mbox{ and } \   \|F(\: \cdot\:+iy)\|_{X}\leq M(C',\varepsilon)\frac{e^{(a+\varepsilon)y}}{|y|^{\kappa(C',\varepsilon)}}, \  \  \ y\in C',
\end{equation}
for any subcone $C'\Subset C$ and $\varepsilon>0$ , then there is $g\in \mathcal{S}'(C^\ast+\overline{B}(a))$ such that $F(z)=\mathcal{L}\{g;z\}$.
\end{corollary}

We also obtain the following corollary, a result of Paley-Wiener type.
\begin{corollary}
\label{bvc6} A necessary and sufficient condition for $f\in\mathcal{D}'_{E'_{\ast}}$ to have $\operatorname*{supp}\hat{f}\subset \overline{B}(a)$ is that $f$ is the restriction to $\mathbb{R}^{n}$ of an entire function $F$ that satisfies $F(\:\cdot\:+iy)\in E'$ for all $y\in\mathbb{R}^{n}$ and the estimate $\sup_{y\in\mathbb{R}^{n}}(1+|y|)^{-m}e^{-a|y|}\|F(\:\cdot\:+iy)\|_{E'}<\infty$ for some $m\geq0$ (or equivalently, $\sup_{y\in\mathbb{R}^{n}}e^{-(a+\varepsilon)|y|}\|F(\:\cdot\:+iy)\|_{E'}<\infty$ for each $\varepsilon>0$).
\end{corollary}
\begin{proof} If $\sup_{y\in\mathbb{R}^{n}}(1+|y|)^{-m}e^{-a|y|}\|F(\:\cdot\:+iy)\|_{E'}<\infty$ for some $m\geq0$, then $\operatorname*{supp}\hat{f}\subset \overline{B}(a)+C^{\ast}$ for every open convex cone and $\bigcap_{C}(\overline{B}(a)+C^{\ast})=\overline{B}(a)$. The other direction can be established as in the proof of Theorem \ref{analyticc}.
\end{proof}

\section{Analytic representations}\label{analytic representations}

The results from  Section \ref{boundary values} enable us to obtain analytic representations of arbitrary elements of $\mathcal D'_{E'_*}$. We need to make use of the following convolution property of $\mathcal D'_{E'_*}$. Recall that $\mathcal{O}'_{C}(\mathbb{R}^{n})$ stands for the space of convolutors of $\mathcal{S}'(\mathbb{R}^{n})$. We have shown in \cite[Prop. 11]{dpv} that $\mathcal D'_{E'_*}$ is closed under convolution with elements of $\mathcal{O}'_{C}(\mathbb{R}^{n})$; furthermore, the convolution mapping $\ast: \mathcal{O}'_{C}(\mathbb{R}^{n}) \times \mathcal D'_{E'_*}\to \mathcal D'_{E'_*}$ is continuous.

Let $C_{1},C_{2},\dots,C_{m}$ be open convex cones of $\mathbb{R}^{n}$. We assume that $\mathbb{R}^{n}=\bigcup_{j=1}^{m}C^\ast_{j}$. For example, the $C_{j}$ might be the $2^{n}$ pairwise disjoint open orthants of $\mathbb{R}^{n}$.

\begin{lemma} \label{arl1} Given $a>0$, there are convolutors $\chi_{1},\chi_{2},\dots,\chi_{m}\in \mathcal{O}'_{C}(\mathbb{R}^{n})$ such that $\delta=\sum_{j=1}^{m}\chi_{j}$ and $\operatorname{supp}\chi_{j}\subset C_{j}^{\ast}+\overline{B}(a)$.
\end{lemma}
\begin{proof} As in \cite[p. 7]{V}, there are $\rho_{1},\dots,\rho_{m}\in C^{\infty}(\mathbb{R}^{n})$ such that
$\operatorname*{supp}\rho_{j}\subset C^{\ast}_{j}+\overline{B}(a)$, $0\leq\rho_{j}\leq 1$, $\rho_{j}(x)=1$ for $x\in C^{\ast}_{j}$, and $\sup_{x\in \mathbb{R}^{n}}|\rho_{j}^{(\alpha)}(x)|\leq M_{\alpha} a^{-|\alpha|}$, $j=1,2,\dots,m$. The distributions $\chi_{\nu}$ given in Fourier side as
$
\hat{\chi}_{\nu}= \rho_{\nu}/(\sum_{j=1}^{m}\rho_{j}) \in \mathcal{O}_{C}(\mathbb{R}^{n})\subset \mathcal{O}_{M}(\mathbb{R}^{n}),$ $\nu=1,2,\dots,m,$
satisfy the requirements.

\end{proof}

We now show that every element of $\mathcal{D}'_{E'_{\ast}}$  can be represented as the sum of boundary values of holomorphic functions.
\begin{theorem}
\label{arth1} Every $f\in \mathcal{D}'_{E'_{\ast}}$ admits the boundary value representation
\begin{equation}\label{areq1}
f=\sum_{j=1}^{m}\underset{y\in C_{j}}{\lim_{y\to 0}}F_{j}(\:\cdot\: +iy) \ \ \ \mbox{strongly in }\mathcal{D}'_{E'_{\ast}},
\end{equation}
where each $F_{j}$ is holomorphic in the tube $T^{C_{j}}$.
\end{theorem}
\begin{proof} Set $f_{j}=\chi_{j}\ast f$ so that $f=\sum_{j=1}^{m}f_{j}$, where $\chi_{1},\dots,\chi_{m}\in \mathcal{O}'_{C}(\mathbb{R}^{n})$ are the distributions from Lemma \ref{arl1}. As mentioned above, by \cite[Prop. 11]{dpv}, each $f_{j}\in \mathcal{D}'_{E'_{\ast}}$. In addition, $\operatorname*{supp} \hat{f}_{j}\subset C_{j}^{\ast}+\overline{B}(a)$. Theorem \ref{analyticc} gives the representation (\ref{areq1}) with $F_{j}(z)=(2\pi)^{-n}\mathcal{L}\{\hat{f}_{j};z\}$.
\end{proof}

The analytic functions $F_{j}$ from Theorem \ref{arth1} of course have the properties (\ref{bveq013}) and (\ref{bveq014}) on the corresponding cone $C_{j}$.


\section{Edge of the wedge theorems}\label{edge of the wedge} 

 This section deals with $\mathcal{D}'_{E'_{\ast}}$-versions of edge of the wedge theorems. Our first results are of Epstein and Bogoliubov type and they are related to the following classes of holomorphic functions on tubes, whose definitions are motivated by Corollary \ref{bvc2}.

\begin{definition}
\label{defa1} Let $V\subseteq\mathbb{R}^{n}$ be open set. The vector space $\mathcal{O}_{E'}(T^{V})$ consists of all holomorphic functions $F$ on the tube $T^{V}=\mathbb{R}^{n}+iV$ satisfying $F(\:\cdot\:+iy)\in E'$ for all $y\in V$ and $\sup_{y\in K}\|F(\:\cdot\:+iy)\|_{E'}<\infty$ for every compact $K\subset V$. The space $\mathcal{O}^{\: b}_{\mathcal{D}'_{E'_{\ast}}}(T^{V})$ is defined as
$$
\mathcal{O}^{\:b}_{\mathcal{D}'_{E'_{\ast}}}(T^{V})=\{F\in \mathcal{O}_{E'}(T^{V}):\: \{F(\:\cdot\:+iy):\: y\in V\} \mbox{ is bounded in } \mathcal{D}'_{E'_{\ast}}\}.
$$
\end{definition}

It should be noticed that if $F\in\mathcal{O}^{\:b}_{\mathcal{D}'_{E'_{\ast}}}(T^{V})$ and $V$ is a truncated cone, then Theorem \ref{thbvw} guarantees that ${\lim_{y\in V\to0}}F(\:\cdot\:+iy)$ exists (strongly) in $\mathcal{D}'_{E'_{\ast}}$. 

We need the following lemma, which is a variant of a result shown by Rudin in \cite[Sect. 3]{rudin1971}.

\begin{lemma}\label{bvl2} Let $V_1$ and $V_2$ be open connected bounded subsets of $\mathbb{R}^{n}$ such that $0\in \partial V_{1}\cap \partial V_{2}$. Set $V=V_1\cup V_2$. Then, any function $F$ that is holomorphic in the tube $T^{V}$, continuous on $T^{V}\cup \mathbb{R}^{n}$, and satisfies $ \sup_{x+iy\in T^{V}} (1+|x|^{2})^{-N/2}|F(x+iy)|<\infty$, for some $N\geq 0$, extends to a function $\tilde{F}$, which is holomorphic in the tube $T^{\operatorname*{ch}(V)}$ and satisfies
\begin{equation}\label{bveq13}
\sup_{x+iy\in T^{\operatorname*{ch}(V)}} \frac{|\tilde{F}(x+iy)|}{(1+|x|^{2})^{nN/2}}\leq M_{N}\sup_{x+iy\in T^{V}} \frac{|F(x+iy)|}{(1+|x|^{2})^{N/2}},
\end{equation}
where the constant $M_{N}$ does not depend on $F$.
\end{lemma}
\begin{remark}
\label{repsteinlemma}
If $V_{1}$ and $V_{2}$ are truncated cones, then the holomorphic function $\tilde{F}$ continuously extends on $T^{\operatorname*{ch}(V)}\cup \mathbb{R}^{n}$, as follows from Epstein's edge of the wedge theorem (cf. \cite[Sect. 11]{rudin1971}).
\end{remark}
\begin{proof}
Applying exactly the same argument as in \cite[Sect. 6.2, p. 122]{stein-weiss1971}, one can show that any function $G$, holomorphic on $T^{V}$, that fulfills the $L^{2}$ conditions
$$
\sup_{y\in V}\int_{\mathbb{R}^{n}} |G(x+iy)|^{2}dx<\infty \ \ \ \mbox{and} \  \  \ \underset{y\in V_1}{\lim_{y\to 0}} G(\: \cdot \:+iy)=\underset{y\in V_{2}}{\lim_{y\to 0}} G(\: \cdot \:+iy), \  \ \mbox{in } L^{2}(\mathbb{R}^{n}),
$$
admits a holomorphic extension $\tilde{G}$ to $T^{\operatorname*{ch}(V)}$. Find $r$ such that $|x|< r$ for all $x\in V$. Let $\lambda>r+1$ and set $Q_{\lambda}(z)=\Pi_{j=1}^{n}(z_{j}+i\lambda)^{N+2}$. The function $G(z)=F(z)/Q_{\lambda}(z)$ satisfies the above two $L^{2}$ conditions and so $\tilde{F}=Q_{\lambda} \tilde{G}$ is the desired holomorphic extension of $F$ to $T^{\operatorname*{ch}(V)}$. We first show (\ref{bveq13}) when $N=0$, which follows if we prove $\tilde{F}(T^{\operatorname*{ch}(V)})\subseteq F(T^{V}\cup \mathbb{R}^{n})$. Indeed, if $\zeta\in \tilde{F}(T^{\operatorname*{ch}(V)})\setminus F(T^{V}\cup\mathbb{R}^{n})$, then $J(z)=1/(F(z)-\zeta)$ would be continuous in $T^{V}\cup \mathbb{R}^{n}$ and holomorphic on $T^{V}$, but this would contradict the fact that $J$ must have a holomorphic extension to the tube $T^{\operatorname*{ch}(V)}$. For general $N$, take again $\lambda>r+1$ and define $F_\lambda (z)=F(z)/\Pi_{j=1}^{n}(z_{j}+i\lambda)^{N}$. Then, if $|F(x+iy)|\leq M(1+|x|^{2})^{N/2}$ for all $x+iy \in T^{V}$, we obtain that $\sup_{x+iy\in \operatorname*{ch}(V)}(1+|x|^{2})^{-nN/2} |\tilde{F}(x+iy)|\leq (\lambda +r)^{nN}\sup_{x+iy\in T^{\operatorname*{ch}(V)}} |\tilde{F}_\lambda (x+iy)|=  (\lambda +r)^{nN}\sup_{x+iy\in T^{V}} |F_\lambda (x+iy)|\leq M(\lambda+r)^{nN} $, which in turn implies the claimed inequality with $M_{N}=(2r+1)^{nN}$.
\end{proof}
We have the following $\mathcal{D}'_{E'_{\ast}}$ edge of the wedge theorem of Epstein type.
\begin{theorem}\label{bvth2} Let $V_1$ and $V_2$ be open connected bounded subsets of $\mathbb{R}^{n}$ with $0\in \partial V_{1}\cap \partial V_{2}$. Set $V=V_{1}\cup V_{2}$. If $F_{1}\in \mathcal{O}^{\:b}_{\mathcal{D}'_{E'_{\ast}}}(T^{V_{1}})$ and $F_{2}\in \mathcal{O}^{\:b}_{\mathcal{D}'_{E'_{\ast}}}(T^{V_{1}})$ have distributional boundary values on $\mathbb{R}^{n}$ and
$$
\underset{y\in V_{1}}{\lim_{y\to 0}}\: F_{1}(\:\cdot\:+iy)=\underset{y\in V_{2}}{\lim_{y\to 0}}\: F_{2}(\:\cdot\:+iy)\  \  \  \mbox{ weakly in }\mathcal{D}'_{E'_{\ast}},
$$
then, there is $F\in \mathcal{O}^{\:b}_{\mathcal{D}'_{E'_{\ast}}}(T^{\operatorname*{ch}(V)})$ such that $F(z)=F_{j}(z)$ for $z\in T^{V_{j}}$, $j=1,2$.
\end{theorem}
\begin{remark} The existence of the limits $\lim_{y\in V_{j}\to 0}\: F_{j}(\:\cdot\:+iy)$ in $\mathcal{D}'_{E'_{\ast}}$, $j=1,2$, is part of the assumptions of Theorem \ref{bvth2}; however, as already mentioned, if $V_{1}$ and $V_{2}$ are truncated cones, such limits automatically exist and in particular $F(\:\cdot\:+iy)$ converges strongly in $\mathcal{D}'_{E'_{\ast}}$ to the common limit as $y\in \operatorname*{ch}(V)$ tends to $0$.
\end{remark}
\begin{proof}
Reasoning as in the proof of Theorem \ref{thbvw} (via Theorem \ref{sequences}), we may assume that $F_{j}$ have continuous extensions to $T^{V_{j}}\cup \mathbb{R}^{n}$ with $F_1(x)=F_{2}(x)$ for $x\in\mathbb{R}^{n}$ and that there is $M$ such that $\|F_{j}(\:\cdot\:+y)\|_{E'}\leq M $ and $|F_{j}(x+iy)|\leq \tilde{M}\omega(x)\leq M(1+|x|^{2})^{\tau/2}$ for $x+iy\in T^{V_{j}}$, $j=1,2$. The pointwise estimate and Lemma \ref{bvl2} imply the existence of $F$, holomorphic in $T^{\operatorname*{ch}(V)}$, such that $F(z)=F_{j}(z)$ for $z\in T^{V_{j}}$, $j=1,2$. It is remains to show that $F(\:\cdot\:+iy)\in E'$ for every $y\in \operatorname*{ch}(V)$ and $\{F(\:\cdot\:+iy):\:y\in \operatorname*{ch}(V)\}$ is bounded in $\mathcal{D}'_{E'_\ast}$. Let $\varphi\in \mathcal{D}(\mathbb{R}^{n})$  with $\|\varphi\|_{E}\leq 1$. Set $G(z):=\int_{\mathbb{R}^{n}}F(t+z)\varphi(t)dt$, $z\in T^{\operatorname*{ch}(V)}$. Then the restriction of $G$ to $T^{V}$ extends continuously to $T^{V}\cup \mathbb{R}^{n}$ and $|G(x+iy)|\leq M (1+|x|^{2})^{\tau/2}$ for $x+iy\in T^{V}$. The inequality (\ref{bveq13}) from Lemma \ref{bvl2} gives $|G(x+iy)|\leq M M_{\tau} (1+|x|^{2})^{n\tau/2}$ for $x+iy\in T^{\operatorname*{ch}(V)}$; in particular $|G(iy)|\leq MM_{\tau}$ for all $y\in \operatorname*{ch}(V)$. Since $\varphi$ is arbitrary and $\mathcal{D}(\mathbb{R}^{n})\hookrightarrow E$, we obtain that $\sup_{y\in \operatorname*{ch}(V)}\|F(\:\cdot\:+iy)\|_{E'}\leq M M_{\tau}$.

\end{proof}

In particular, we have the ensuing corollary on analytic continuation; the case $C_{2}=-C_{1}$ is an edge of the wedge theorem of Bogoliubov type.

\begin{corollary}
\label{bvc4} Let $C_{1}$ and $C_{2}$ be open cones such that $\operatorname*{int}(C_{1}^{\ast})\cap \operatorname*{int}(C_{2}^{\ast})=\emptyset$ and let $r_1,r_2>0$. Set $V=C_1(r_{1})\cup C_{2}(r_2)$. If $F_{j}\in \mathcal{O}^{\: b}_{\mathcal{D}'_{E'_{\ast}}}(T^{C_j(r_j)})$, $j=1,2$, are such that
$$
\underset{y\in C_{1}}{\lim_{y\to 0}}\: F_{1}(\:\cdot\:+iy)=\underset{y\in C_{2}}{\lim_{y\to 0}}\: F_{2}(\:\cdot\:+iy)\  \  \  \mbox{ in }\mathcal{D}'_{E'_{\ast}},
$$
then $F_{1}$ and $F_{2}$ can be glued together as a holomorphic function through $\mathbb{R}^{n}$; more precisely, the domain $T^{\operatorname*{ch}(V)}$ of their holomorphic extension $F \in \mathcal{O}^{\: b}_{\mathcal{D}'_{E'_{\ast}}}(T^{\operatorname*{ch}(V)})$ contains a tube $\mathbb{R}^{n}+i\{y\in\mathbb{R}^{n}:\: |y|< r\}$.
\end{corollary}
\begin{proof} The condition implies that the cone $C=\operatorname*{ch}(C_1\cup C_2)$ contains a line, and therefore the origin as interior point.
\end{proof}

Our next result is an edge of the wedge theorem of Martineau type \cite{martineau1995,Morimoto}, it is related to the classes of holomorphic functions on wedges introduced in the next definition.

\begin{definition}
\label{defa2} Let $C$ be an open convex cone and $a\geq 0$.
\begin{itemize}
\item [$(i)$] If $C\neq \mathbb{R}^{n}$, we define
$
\mathcal{O}_{\mathcal{D}'_{E'_{\ast}}}^{\:a,\:\textnormal{exp}}(T^{C})$ as the space of all holomorphic functions $F\in\mathcal{O}_{E'}(T^{C})$ such that there is $\kappa\geq 0$ such that  for every $\varepsilon>0$
$$\sup_{y\in C}e^{-(a+\varepsilon)|y|}\left(1+\frac{1}{d_{C}(y)}\right)^{-\kappa}\|F(\:\cdot\:+iy)\|_{E'}<\infty.$$
\item [$(ii)$] When $C=\mathbb{R}^{n}$, the space $\mathcal{O}_{\mathcal{D}'_{E'_{\ast}}}^{\:a,\:\textnormal{exp}}(\mathbb{C}^{n})$ consists of all $F\in\mathcal{O}_{E'}(\mathbb{C}^{n})$ such that for every $\varepsilon>0$
$$
\sup_{y\in \mathbb{R}^{n}}e^{-(a+\varepsilon)|y|}\|F(\:\cdot\:+iy)\|_{E'}<\infty.
$$
We also use the notation $\mathcal{O}_{E'}^{\:a,\:\textnormal{exp}}(\mathbb{C}^{n}):=\mathcal{O}_{\mathcal{D}'_{E'_{\ast}}}^{\:a,\:\textnormal{exp}}(\mathbb{C}^{n})$ for this space.
\end{itemize}
\end{definition}
\smallskip

Observe that Remark \ref{remarkemb1} and Theorem \ref{th structure} allow us to conclude that $\mathcal{O}_{\mathcal{D}'_{E'_{\ast}}}^{\:a,\: \textnormal{exp}}(T^{C})\subseteq \mathcal{O}_{\mathcal{B}'_{\omega}}^{\:a,\:\textnormal{exp}}(T^{C})$. In particular, every element of $\mathcal{O}_{E'}^{\:a,\:\textnormal{exp}}(\mathbb{C}^{n})$ is actually an entire function of exponential type.

According to Theorem \ref{analyticc} (cf. Corollary \ref{bvc6} for the case $C=\mathbb{R}^{n}$), every element $F\in \mathcal{O}^{\textnormal{a,exp}}_{\mathcal{D}_{E'_\ast}}(T^{C})$ is completely determined by its boundary value distribution, which we denote by $\operatorname*{bv} (F):=\lim_{y\in C\to 0} F(\:\cdot\:+iy)\in\mathcal{D}'_{E'_{\ast}}$. In the next theorem each $C_{j}$ is an open convex cone. Note that it considerably improves earlier results by Carmichael \cite{carmichael85}.

\begin{theorem}\label{arth3} Let $F_{j}\in \mathcal{O}^{\:a,\:\textnormal{exp}}_{\mathcal{D}'_{E'_{\ast}}}(T^{C_{j}})$, $j=1,2,\dots,k$, and let $\varepsilon>0$. Set $C_{j,\nu}=\operatorname*{ch}(C_{j}\cup C_{\nu})$ and $\widetilde{C}_{j}=\bigcap_{\nu\neq j}C_{j,\nu}$. If $\sum_{j=1}^{k}\operatorname*{bv} (F_{j})=0$, then for each $j$ there are $G_{j,\nu}\in \mathcal{O}^{\:a+\varepsilon,\:\textnormal{exp}}_{\mathcal{D}'_{E'_{\ast}}}(T^{C_{j,\nu}})$ such that $F_{j}=\sum_{\nu=1}^{k}G_{j,\nu}$. In particular, each $F_{j}$ has a holomorphic extension that belongs to $\mathcal{O}^{\:a+\varepsilon,\:\textnormal{exp}}_{\mathcal{D}'_{E'_{\ast}}}(T^{\widetilde{C}_{j}})$. The $G_{j,\nu}$ may be chosen such that $G_{\nu,j}=-G_{j,\nu}$.
\end{theorem}

\begin{proof} If some of the $C_{\nu}$ are $\mathbb{R}^{n}$, the corresponding terms in the sum can be absorbed into others. We may therefore assume that all $C_{1},\dots,C_{k}$ are open convex cones with $C_{\nu}\neq \mathbb{R}^{n}$. We can find $g_{\nu}$ such that $F_{\nu}(z)=\mathcal{L}\{g_{\nu};z\}$, with $\operatorname*{supp}g_{\nu}\subset C_{\nu}^{\ast}+\overline{B}(a)$ and $\hat{g}_{\nu}\in\mathcal{D}'_{E'_\ast}$. Find $\rho_{\nu}$ with bounded partial derivatives of any order such that $\operatorname*{supp}\rho_{\nu}\subseteq C_{\nu}^{\ast}+B(a+\varepsilon)$ and $\rho_{\nu}(x)=1$ for $x\in\operatorname*{supp}g_{\nu}$. Then, $g_{j}=-\sum_{\nu\neq j} \rho_{j}g_{\nu}$. Setting $G_{j,\nu}\in\mathcal{O}^{\:a+\varepsilon,\:\textnormal{exp}}_{\mathcal{D}'_{E'_{\ast}}}(T^{C_{j,\nu}})$ as the Laplace transform of $-\rho_{j}g_{\nu}$, we obtain $F_{j}=\sum_{\nu\neq j}G_{j,\nu}$. It remains to be shown that the $G_{j,\nu}$ may be chosen such that $G_{\nu,j}=-G_{j,\nu}$.
We proceed by induction over the number of summands. The cases $k=1,2$ are trivial. Assume that such a choice is possible for $k$. If $\sum_{j=1}^{k+1}\operatorname*{bv} (F_{j})=0$, from what we have shown we can write $F_{k+1}=\sum_{j=1}^{k}G_{k+1,\nu}$ where $G_{k+1,\nu}\in \mathcal{O}^{\:a+\varepsilon,\:\textnormal{exp}}_{\mathcal{D}'_{E'_{\ast}}}(T^{C_{k+1,\nu}})$. Thus, $\sum_{j=1}^{k}\operatorname*{bv} (G_{k+1,\nu}+F_{\nu})=0$. By the inductive hypothesis, we get that there are $G_{j,\nu}\in \mathcal{O}_{\mathcal{D}'_{E'_{\ast}}}^{\:a+\varepsilon,\:\textnormal{exp}}(T^{C_{j,\nu}})$ such that $G_{j,\nu}=-G_{\nu,j}$, $1\leq j,\nu\leq k$, and $F_{j}+G_{k+1,j}=\sum_{\nu=1}^{k}G_{j,\nu}$. The property is then satisfied if we define $G_{j,k+1}:=-G_{k+1,j}$ and $G_{k+1,k+1}=0$.
\end{proof}

\section{Analytic representation of $\mathcal{D}'_{E'_{\ast}}$ -- boundary value isomorphism}
\label{boundary isomorphism}
We have shown in Section \ref{analytic representations} that every element of $\mathcal{D}'_{E'_{\ast}}$ admits a representation as a sum of boundary values of holomorphic functions. We now give in Subsection \ref{hyperfunctions} an isomorphism between $\mathcal D'_{E'_*}$ and a quotient space of analytic functions in the spirit of the hyperfunction theory (see \cite{Morimoto,Kaneko} for hyperfunctions). Such an isomorphism will be accomplished through our version of Martineau's edge of the wedge theorem (Theorem \ref{arth3}). The results are then specialized to the one-dimensional case in Subsection \ref{one-dimensional case}.

\subsection{The boundary value isomorphism $\mathcal{D}'_{E'_{\ast}}\cong\mathcal{D}b_{E'}^{\textnormal{ exp}}(\mathbb R^n)$}\label{hyperfunctions}

We will use our results to represent the space $\mathcal{D}'_{E'_{\ast}}$ as a quotient space of analytic functions. We introduce suitable spaces of analytic functions. For an open convex cone $C$, we set (cf. Definition \ref{defa2})
$$
\mathcal{O}_{\mathcal{D}'_{E'_{\ast}}}^{\textnormal{ exp}}(T^{C})=\bigcup_{a\geq0} \mathcal{O}_{\mathcal{D}'_{E'_{\ast}}}^{\:a,\:\textnormal{exp}}(T^{C}).
$$
We also use the notation
$
\mathcal{O}_{E'}^{\textnormal{ exp}}(\mathbb{C}^{n})=\mathcal{O}_{\mathcal{D}'_{E'_{\ast}}}^{\textnormal{ exp}}(\mathbb{C}^{n})$.

We consider $\bigoplus_{C}{\mathcal {O}}_{\mathcal{D}_{E'_\ast}}^{\textnormal{ exp}}(T^{C}),
$
where $C$ runs over all open convex cones of $\mathbb{R}^{n}$ (so that $\mathcal{O}_{E'}^{\textnormal{ exp}}(\mathbb{C}^{n})$ is a term of the direct sum), and its subspace $\mathcal{N}_{\mathcal{D}_{E'_\ast}}^{\textnormal{ exp}}$ generated by all elements of the form
$F_{1}+F_2-F_3$, where $F_j\in {\mathcal {O}}_{\mathcal{D}'_{E'}}^{\textnormal{ exp}}(T^{C_{j}})$, $j=1,2,3,$ are such that
$C_3\subseteq C_1\cap C_2$, and $F_1(z)+F_2(z)=F_3(z)$ for $z\in C_{3}$. We remark that some of the three functions may be identically zero. Next, we
define the quotient vector space
$$\mathcal{D}b_{E'}^{\textnormal{ exp}}(\mathbb{R}^{n})=\left(\bigoplus_{C}{\mathcal {O}}_{\mathcal{D}_{E'_\ast}}^{\textnormal{ exp}}(T^{C})\right)/\mathcal{N}^{\textnormal{ exp}}_{\mathcal{D}_{E'_\ast}}.$$
The equivalence class of $\sum_{j=1}^{k}F_{j}\in \bigoplus_{C}\mathcal{O}_{\mathcal{D}_{E'_\ast}}^{\textnormal{ exp}}(T^{C})$ is denoted by $[\sum_{j=1}^{k}F_{j}]= \sum_{j=1}^{k}[F_{j}]$.

The mappings $\operatorname*{bv}:\mathcal{O}^{\textnormal{ exp}}_{\mathcal{D}_{E'_\ast}}(T^{C})\to \mathcal{D}'_{E'_{\ast}}$ clearly induce a well-defined boundary value mapping
\begin{equation}
\label{bvmeq}
\operatorname*{bv}:\mathcal{D}b_{E'}^{\textnormal{ exp}}(\mathbb{R}^{n})\to \mathcal{D}'_{E'_{\ast}},
\end{equation}
namely,
$\operatorname*{bv}(\sum_{j=1}^{k}[F_{j}])= \sum_{j=1}^{k}\operatorname*{bv} (F_{j}) \in \mathcal{D}'_{E'_{\ast}}$. Combining our previous results, we obtain:

\begin{theorem}\label{arth2} The boundary value mapping $\operatorname*{bv}:\mathcal{D}b_{E'}^{\textnormal{ exp}}(\mathbb{R}^{n})\to \mathcal{D}'_{E'_{\ast}}$ is a bijection.
\end{theorem}

Indeed, that (\ref{bvmeq}) is surjective follows at once from Theorem \ref{arth1}, whereas the injectivity is a consequence of Theorem \ref{arth3}.

\subsection{The one-dimensional case} \label{one-dimensional case} Assume that the dimension $n=1$. The previous construction significantly simplifies if we take into account the natural orientation of the real line. Consider first
$$
\mathcal{O}_{\mathcal{D}'_{E'_{\ast}}}^{\textnormal{ exp}}(\mathbb{C}\setminus \mathbb{R})=\{F\in \mathcal{O}_{E'}(\mathbb{C}\setminus \mathbb{R}):\: F_{|\mathbb{R}\pm i(0,\infty)}\in \mathcal{O}_{\mathcal{D}'_{E'_{\ast}}}^{\textnormal{ exp}}(\mathbb{R}\pm i(0,\infty))\}.
$$
If we replace the boundary value mapping by a jump across $\mathbb{R}$ mapping, we obtain that
$$
\mathcal{D}'_{E'_{\ast}}\cong \mathcal{O}^{\textnormal{ exp}}_{\mathcal{D}'_{E'_{\ast}}}(\mathbb{C}\setminus \mathbb{R})/\mathcal{O}_{E'}^{\textnormal{ exp}}(\mathbb{C});
$$
the isomorphism being realized by the mapping $\mathcal{O}^{\textnormal{ exp}}_{\mathcal{D}'_{E'_{\ast}}}(\mathbb{C}\setminus \mathbb{R})/\mathcal{O}_{E'}^{\textnormal{ exp}}(\mathbb{C})\to \mathcal{D}'_{E'_{\ast}}$ given by $[F]\mapsto \lim_{y \to 0^{+}} (F(\:\cdot\:+iy)- F(\:\cdot\:-iy))$.

We may also give another version of the quotient representation. Let $\Omega$ be a neighborhood of the real line of the form $\Omega=\mathbb{R}+iI$, where $I$ is an open interval containing 0. Set
$$
\mathcal{O}_{\mathcal{D}'_{E'_{\ast}}}(\Omega\setminus \mathbb{R})=\{F\in \mathcal{O}_{E'}(\Omega\setminus \mathbb{R}):\: (\forall I'\Subset I)(\exists \kappa)( \sup_{y\in I'\setminus\{0\}}|y|^{\kappa}\|F(\:\cdot\:+iy)\|_{E'}<\infty )\}.
$$
Then, in view of Theorem \ref{thbvw} and the edge of the wedge theorem of Epstein type (Theorem \ref{bvth2}), the jump across $\mathbb{R}$ mapping produces the isomorphism
$$
\mathcal{D}'_{E'_{\ast}}\cong \mathcal{O}_{\mathcal{D}'_{E'_{\ast}}}(\Omega\setminus \mathbb{R})/\mathcal{O}_{E'}(\Omega).
$$

\section{Heat kernel characterization}\label{heat}

We now turn our attention to the characterization of elements $\mathcal{D}'_{E'_{\ast}}$ as boundary values of solutions to the heat equation on $\mathbb{R}^{n}\times (0,t_{0})$. Given $f\in\mathcal{D}'(\mathbb{R}^{n})$, we consider the Cauchy problem for the heat equation
\begin{equation}
\label{heateq1}
\partial_{t}U- \Delta_{x} U=0, \ \ \ (x,t)\in\mathbb{R}^{n}\times (0,t_{0}),
\end{equation}
with initial value
\begin{equation}
\label{heateq2}
\lim_{t\to0^{+}} U(\:\cdot\:,t)= f \ \ \ \mbox{in }\mathcal{D}'(\mathbb{R}^{n}).
\end{equation}
Observe that under certain bounds over $U$, such as \cite{C}
\begin{equation*}
|U(x,t)|\leq M \exp{\left(\left(\frac{a}{t}\right)^{\alpha}+ a|x|^{2}\right)} \ \ \ (0<\alpha<1, \ a>0),
\end{equation*}
one can ensure uniqueness of the solution $U$ and, in such a case, $U$ is determined via convolution with the heat kernel: $U(x,t)=(4\pi t)^{-n/2}\left\langle f(\xi),e^{-\frac{|\xi-x|^{2}}{4t}}\right\rangle$.

In order to move further, we need a characterization of $\mathcal{D}'_{E'_{\ast}}$ in terms of norm growth bounds on convolutions with an approximation of the unity. For it, we employ the useful concept of the $\phi-$transform \cite{DZ2,PV,PV2014,estrada-vindasI}, which is defined as follows. Let $\phi\in\mathcal{S}(\mathbb{R}^{n})$ be such that $\int_{\mathbb{R}^{n}}\phi(x) d x=1$. The $\phi-$transform of $f\in\mathcal{S}'(\mathbb{R}^{n})$ is the smooth function
$$
F_\phi f(x,t)=\left\langle  f(x+t\xi),\phi\left(\xi\right)\right\rangle= (f\ast \check{\phi}_{t})(x), \ \ \ (x,t)\in \mathbb{R}^n\times\mathbb{R}_+.
$$
\begin{theorem}\label{phi transform}
A distribution $f\in \mathcal{S}'(\mathbb{R}^n)$ belongs to $\mathcal{D}_{E'_{\ast}}'$ if and only if $F_{\phi}f(\:\cdot\:, t)\in E'$ for all $t\in (0,t_{0})$ and there are constants $k\in \mathbb{N}$ and $M>0$ such that
\begin{equation}
\label{eq11}
\|F_\phi f(\:\cdot\:,t)\|_{E'}\leq\frac{M}{t^k}, \ \ \ t\in(0,t_{0}).
\end{equation}
In such a case,
\begin{equation}
\label{eq12}
\lim_{t\to0^{+}}F_\phi f(\:\cdot\:,t)= f \ \ \ \mbox{strongly in }\mathcal{D}'_{E'_{\ast}}.
\end{equation}
\end{theorem}
\begin{proof} The relation (\ref{eq12}) follows by combining $(iii)$ of Theorem \ref{th structure} with (\ref{eqapprox}).
Let $f\in \mathcal{D}_{E'_{\ast}}'$. Write $f$ as in (\ref{eq:representation}). By (\ref{estimateconv2}), for $t\in(0,t_{0}]$,

\begin{align*}
||F_{\phi}f(\:\cdot\:,t)||_{E'} &\leq \sum_{|\alpha|\leq N}|| f_{\alpha}\ast (\check{\phi}_{t})^{(\alpha)}||_{E'}\\
&
\leq \frac{M'}{t^{N}}\sum_{|\alpha|\leq N}|| f_{\alpha}||_{E'}\int_{\mathbb{R}^{n}}|\phi^{(\alpha)}(x)|\:\omega(tx)dx\leq \frac{M}{t^{N}}.
\end{align*}

Conversely, assume (\ref{eq11}).  Let $\varphi\in \mathcal{S}(\mathbb{R}^{n})$ be also such that $\int_{\mathbb{R}^{n}}\varphi(x)dx=1$. Setting $\phi_{1}=\phi \ast \varphi$, we have that $F_{\phi_{1}}f(x,t)=(F_{\phi}f(\:\cdot\:,t)\ast \check{\varphi}_{t})(x)$, and so $F_{\phi_{1}} f(\:\cdot\:,t)\in E'\ast \mathcal{S}(\mathbb{R}^{n}) \subset E'_{\ast}$ for each $t\in(0,t_{0})$. We will use the theory of (Tauberian) class estimates from \cite{DZ2,PV}. Set $\mathbf{f}=\iota (f)\in \mathcal{S}'(\mathbb{R}^{n},\mathcal{S}'(\mathbb{R}^{n}))$, where $\mathbf{f}$ is given by (\ref{eq:8}). By $(ii)$ of Theorem \ref{th structure} and (\ref{eqrange}), it is enough to show that $\mathbf{f}\in\mathcal{S}'(\mathbb{R}^{n},E'_{\ast})$ (cf. (\ref{embedding})). The $\mathcal{S}'(\mathbb{R}^{n})$-valued $\phi_{1}-$transform of $\mathbf{f}$ is the vector-valued distribution $F_{\phi_{1}}\mathbf{f}:\mathbb{R}^{n}\times \mathbb{R}_{+}\to \mathcal{S}'(\mathbb{R}^{n})$ given by
$
F_{\phi_1}\mathbf{f}(x,t)=T_{x}F_{\phi_{1}}f(\cdot,t)\in \mathcal{S}'(\mathbb{R}^{n}).
$ From what has been shown we have that $F_{\phi_1}\mathbf{f}(x,t)\in E'_{\ast}\subset \mathcal{S}'(\mathbb{R}^{n})$ for all $(x,t)\in \mathbb{R}^{n}\times (0,t_{0})$ and, by property $(d)$ applied to $E'_{\ast}$ (cf. Section \ref{the space E}), we get that the the mapping $\mathbb{R}^{n}\to E'_{\ast}$ given by $x\mapsto F_{\phi_1}\mathbf{f}(x,t)$ is continuous for each fixed $t\in(0,t_{0})$. Furthermore, using the fact that $E'_{\ast}$ is a Banach modulo over $L^{1}_{\check{\omega}}$, we conclude that
\begin{align*}
||F_{\phi_1}\mathbf{f}(x,t)||_{E'}&=||T_{x}F_{\phi_1}f(\cdot,t)||_{E'}\\
&
\leq \omega(x) ||F_{\phi}f(\: \cdot\:,t)||_{E'} \int_{\mathbb{R}^{n}}|\varphi(\xi)|\omega(t\xi)d\xi \leq \tilde{M}\frac{(1+|x|)^{\tau}}{t^{k}},
\end{align*}
for all $(x,t)\in \mathbb{R}^{n}\times (0,t_{0})$. But, as shown in \cite{DZ2} (see also \cite[Sect. 7]{PV}), the very last estimate is necessary and sufficient for $\mathbf{f}\in\mathcal{S}'(\mathbb{R}^{n},E'_{\ast})$. This completes the proof.
\end{proof}

After this preliminaries, we are ready to state and prove the heat kernel characterization of $\mathcal{D}'_{E'_{\ast}}$.

\begin{theorem} \label{heat th1} Let $f\in\mathcal{D}'(\mathbb{R}^{n})$. Then, $f\in\mathcal{D}'_{E'_{\ast}}$ if and only if there is a solution $U$ to the Cauchy problem (\ref{heateq1}) and (\ref{heateq2}) that satisfies
\begin{equation}
\label{heateq3}
U(\:\cdot\:,t)\in E' \ \ \ \mbox{for all } t\in(0,t_{0})
\end{equation}
and there are constants $M>$ and $k\geq 0$ such that
\begin{equation}
\label{heateq4}
||U(\:\cdot\:,t)||_{E'}\leq \frac{M}{t^{k}}, \ \ \ t\in(0,t_{0}).
\end{equation}
In such a case,
\begin{equation}
\label{heateq5}
\lim_{t\to0^{+}}U(\:\cdot\:,t)=f \ \ \ \mbox{strongly in } \mathcal{D}'_{E'_{\ast}}.
\end{equation}
\end{theorem}
\begin{proof}
If $f\in \mathcal{D}_{E'_{\ast}}'$, then $U(x,t)=F_{\phi}f(x,\sqrt{t})$ with $\phi(\xi)= (4\pi)^{-n/2}e^{-|\xi|^{2}/4}$ satisfies (\ref{heateq1}), (\ref{heateq2}), and (\ref{heateq3})--(\ref{heateq5}), as follows from Theorem \ref{phi transform}. Conversely, assume that (\ref{heateq1}), (\ref{heateq2}), (\ref{heateq3}), and (\ref{heateq4}) hold for $U$. Applying Theorem \ref{th structure}, we conclude that $U$ can be written as
\begin{equation}
\label{heateq6}
U(x,t)=\sum_{|\alpha|\leq N} \partial ^{\alpha}_x U_{\alpha}(x,t), \  \  \  (x,t)\in \mathbb{R}^{n}\times(0,t_{0}),
\end{equation}
where each $U_{\alpha}$ has the form $U_{\alpha}(x,t)= (U(\:\cdot\:,t)\ast \check{\varrho}_{\alpha})(x)=\int_{\mathbb{R}^{n}}U(x+\xi,t)\varrho_{\alpha}(\xi)d\xi$, with $\varrho_{\alpha}\in E$ being compactly supported and continuous. Each $U_{\alpha}$ is also a solution to the heat equation on $\mathbb{R}^{n}\times(0,t_{0})$, and it satisfies $U_{\alpha}(\:\cdot\:,t)\in E'_{\ast}$ for all $t\in(0,t_{0})$, the $E'$-norm estimate
\begin{equation}
\label{heateq7}
\|U_{\alpha}(\:\cdot\:,t)\|_{E'}\leq \frac{M\|\varrho_{\alpha}\|_{1,\omega}}{t^{k}}, \ \ \  t\in(0,t_{0}),
\end{equation}
and the pointwise estimate
\begin{equation}
\label{heateq8}
|U_{\alpha}(x,t)|\leq M \|\varrho_{\alpha}\|_{E}\frac{\omega(x)}{t^{k}}\leq M_{\alpha}\frac{(1+|x|)^{\tau}}{t^{k}}, \  \  \  (x,t)\in \mathbb{R}^{n}\times (0,t_0).
\end{equation}
Using the pointwise estimate (\ref{heateq8}) and applying Matsuzawa's heat kernel characterization of $\mathcal{S}'(\mathbb{R}^{n})$ \cite{M1990}, one concludes the existence of $f_{\alpha}\in \mathcal{S}'(\mathbb{R}^{n})$ such that $\lim_{t\to0^{+}}U_{\alpha}(\:\cdot\:,t)=f_{\alpha}$ in $\mathcal{S}'(\mathbb{R}^{n})$, for each $|\alpha|\leq N$. The uniqueness criterion for solutions to the heat equation \cite{C} yields $U_{\alpha}(x,t)=F_{\phi}f_{\alpha}(x,\sqrt{t})$ with again $\phi(\xi)= (4\pi)^{-n/2}e^{-|\xi|^{2}/4}$. The $E'$-norm estimate (\ref{heateq7}) thus implies that each $f_{\alpha}\in \mathcal{D}'_{E'_{\ast}}$ (again by Theorem \ref{phi transform}). Finally, by (\ref{heateq6}), we get $f=\sum_{|\alpha|\leq N}f_{\alpha}^{(\alpha)}\in \mathcal{D}'_{E'_{\ast}}$.
\end{proof}

Theorem \ref{heat th1} is complemented by the ensuing result, whose proof was already given within that of Theorem \ref{heat th1}.

\begin{corollary} \label{heat c1} Let $U$ be a solution to the heat equation (\ref{heateq1}) that satisfies (\ref{heateq3}) and the estimate (\ref{heateq4}). Then, there is a distribution $f\in\mathcal{D}_{E'_{\ast}}'$ such that (\ref{heateq5}) holds. Moreover, $U$ is uniquely determined by $f$.
\end{corollary}

We end this article with the following corollary, the proof is analogous to that of Corollary \ref{bvc5}.

\begin{corollary}
\label{heat c2} Let $X\subset \mathcal{S}'(\mathbb{R}^{n})$ be a Banach space for which the inclusion mapping $X\to \mathcal{S}'(\mathbb{R}^{n})$ is continuous. If $U$ is a solution to the heat equation (\ref{heateq1}) that satisfies $U(\:\cdot\:,t)\in X$ for every $t\in(0,t_{0})$ and the estimate $\sup_{t\in(0,t_{0})} t^{k}\|U(\:\cdot\:,t)\|_{X}<\infty$ for some $k\geq0$, then $\lim_{t\to0^{+}}U(\:\cdot\:,t)$ exists in $\mathcal{S}'(\mathbb{R}^{n})$.
\end{corollary}



\end{document}